%% file: Main_new.tex
\documentclass[reqno,a4paper]{article}
 \usepackage{amsmath,amssymb,amsthm,amsfonts,epsfig,enumerate}
 \usepackage{mathtools}
 \usepackage{multicol}
 \usepackage[margin=0.99 in]{geometry}
 \usepackage{color}
 \usepackage{cmap,mathtools}
 \usepackage{enumerate}
 \usepackage{cite}
 \usepackage[utf8]{inputenc}
 \usepackage{fancyhdr}
\usepackage[parfill]{parskip}
\usepackage{graphicx}
\usepackage{booktabs}

 \usepackage{hyperref}
 \hypersetup{linkcolor=blue, colorlinks=true ,citecolor = red}

\newcommand{\dx}{{\mathrm{d}x}}
\newcommand{\dm}{{\mathrm{d}\mu}}
\newcommand{\dy}{{\mathrm{d}y}}
\newcommand{\dxy}{{\mathrm{d}x\mathrm{d}y}}
\newcommand{\D}{{\mathcal{D}}}
\usepackage{authblk}
\include{preamp}
\providecommand{\keywords}[1]
{
  \small	
  \textbf{\textit{Keywords---}} #1
}
\providecommand{\msc}[1]
{
  \small	
  \textbf{\textit{Mathematics Subject Classification (2010)---}} #1
}
\title{Characterizations of compactness and weighted eigenvalue
problem {associated with} fractional Hardy-type inequalities}
 \author{Ujjal Das, Rohit Kumar and Abhishek Sarkar\thanks{Corresponding author.}}

\date{}
\usepackage[hyperpageref]{backref}

\begin{document}
 \maketitle
\begin{abstract}
In this article, we consider the following fractional {Hardy-type} inequality:
\begin{align} \label{Fractional Hardy_abst}
  \int_{\mathbb{R}^N} |w(x)||u(x)|^p  \mathrm{d}x \leq C  \int_{\mathbb{R}^N \times \mathbb{R}^N} \frac{|u(x)-u(y)|^p}{|x-y|^{N+sp}} \dxy:= \|u\|_{s,p}^p\,, \ \forall u \in \mathcal{D}^{s,p}(\mathbb{R}^N), 
\end{align} where $0<s<1<p<\frac{N}{s}$,  and $\D^{s,p}(\mathbb{R}^N)$ is the completion of $C_c^{\infty}(\mathbb{R}^N)$ with respect to the {norm} $\|\cdot\|_{s,p}$. We denote the space of all admissible {weight function} $w$ in \eqref{Fractional Hardy_abst} by  $\mathcal{H}_{s,p}(\mathbb{R}^N)$. Maz'ya-type characterization helps us to define a Banach function norm on $\mathcal{H}_{s,p}(\mathbb{R}^N)$.  Using the Banach function space structure and the concentration compactness type arguments, we provide several characterizations for the compactness of the map ${W}(u)= \int_{{\R^N}} |w| |u|^p \dx$ on $\mathcal{D}^{s,p}(\mathbb{R}^N)$. In particular, we prove that ${W}$ is compact on $\mathcal{D}^{s,p}(\mathbb{R}^N)$ if and only if $w \in \mathcal{H}_{s,p,0}(\mathbb{R}^N):=\overline{C_c(\mathbb{R}^N)} \ \mbox{in} \ \mathcal{H}_{s,p}(\mathbb{R}^N)$. Further, we study the following {weighted} eigenvalue problem: \begin{equation*} 
    (-\Delta_{p})^{s}u = \lambda w(x) |u|^{p-2}u ~~\text{in}~\mathbb{R}^{N},
\end{equation*}  where $(-\Delta_{p})^{s}$ is the fractional $p$-Laplace operator and $w = w_{1} - w_{2}~\text{with}~ w_{1},w_{2} \geq 0,$ is such that  $ w_{1} \in \mathcal{H}_{s,p,0}(\R^N)$ and $w_{2} \in L^1_{\mathrm{loc}}(\R^N)$.
 \end{abstract} 

\medskip
\msc{26D10, 31B15, 35A15, 35R11} 

\maketitle

\input{Intro_new}
\input{Prelim_new}
\input{Results_new}


\section*{Acknowledgements}
 {The project was mostly completed when U.D. was in Technion-Israel Institute of Technology, as a postdoc.} U.D. acknowledges the support of the Israel Science Foundation \!(grant $637/19$) founded by the
Israel Academy of Sciences and Humanities. U.D. was also partially supported by a fellowship from the Lady Davis Foundation. {R.K.} wants to thank the support of the CSIR fellowship, File No. 09/1125(0016)/2020--EMR--I, for his Ph.D. work. A.S. was supported by the DST-INSPIRE grant DST/INSPIRE/04/2018/002208.

\bibliographystyle{plain}
\bibliography{ref}


\noi {\bf	 Ujjal Das }\\   School of Mathematical Sciences, National Institute of Science Education and Research (NISER), Bhubaneswar,\\ 752050 Odisha, India. \\
	 {\it Email}: udas@niser.ac.in, getujjaldas@gmail.com\\
\noi{\bf    Rohit Kumar}\\ Department of Mathematics, Indian Institute of Technology Jodhpur\\ Rajasthan 342030, India.\\ 
   {\it Email: rohit1.iitj@gmail.com, kumar.174@iitj.ac.in}\\
	\noi{\bf   Abhishek Sarkar}\\ Department of Mathematics, Indian Institute of Technology Jodhpur\\ Rajasthan 342030, India.\\  
 {\it    Email: abhisheks@iitj.ac.in}
		
\end{document}

%% file: preamp.tex
 \usepackage{amsmath,amssymb,amsthm,amsfonts,epsfig,enumerate}
 \usepackage{mathtools}
 \usepackage[dvipsnames]{xcolor}
\newtheorem{definition}{Definition}[section]
\newtheorem{theorem}[definition]{Theorem}

\newtheorem{corollary}[definition]{Corollary}
\newtheorem{remark}[definition]{Remark}
\newtheorem{proposition}[definition]{Proposition}
\newtheorem{lem}[definition]{Lemma}
\newtheorem{lemma}[definition]{Lemma}

\numberwithin{equation}{section}

\DeclarePairedDelimiter\norm{\lVert}{\rVert}
\makeatletter

\newcommand{\Ga} {\Gamma}

\newcommand{\Om} {\Omega}

\newcommand{\noi} {\noindent}
\newcommand{\var} {\epsilon}
\newcommand{\ra} {\rightarrow}

\newcommand{\wra} {\rightharpoonup}
\newcommand{\wrastar} {\overset{\ast}{\rightharpoonup}}

\DeclareMathAlphabet{\mathpzc}{T1}{pzc}{m}{it}

\def\w{{\widetilde w}}

\def\dx{{\,\rm d}x}
\def\dt{{\rm d}t}

\def\cp{{\rm Cap}_{s,p}}

\def \tk {\textcolor{black}}

\def\sb2{{{\mathcal D}^{1,2}_0(B_1^c)}}

\def\w2r{{{ W}^{2,2}(\R^N)}}

\def\d2{{{\mathcal D}^{2,2}_0(\Om)}}

\def\C{{\mathcal C}}
\def\D{{\mathcal D}}

\def\R{{\mathbb R}}
\def\N{{\mathbb N}}
\def\F{{\mathcal F}}
\def\({{\Big(}}
\def\){{\Big)}}

\def\ws2{{\F_{\frac{N}{2}}}}
\def\L2{{ L^{1,\;\infty}(\log L)^2}}
\def\dx{{\rm d}x}

\def\dt{{\rm d}t}
\def\l2{\mathcal M\log L}

\def\c1Loc{{\C_{loc}^1}}

%% file: Intro_new.tex
\section{Introduction}

For $p\in (1,N)$, the Beppo Levi space $\mathcal{D}^{1,p}(\R^N)$ is the completion of  $C_c^{\infty}(\R^N)$ with respect to the norm $ \norm{u}_{{1,p}} :=\left[ \int_{\R^N}|\nabla u|^p \dx \right]^ \frac{1}{p}$.  Let us first  recall the following classical Hardy inequality: 
  \begin{equation} \label{CHS}
  \int_{\R^N} \frac{1}{|x|^p} |u|^p\ \dx \leq \displaystyle \left(\frac{p}{N-p} \right)^p \int_{\R^N} |\nabla u|^p \, \dx, \ \forall u \in \mathcal{D}^{1,p}(\R^N) \,.
  \end{equation}
The one-dimensional Hardy inequality was proved by Hardy  \cite{Hardy} in 1920. 
\tk{There has been considerable interest in} identifying more general weight function $w \in L^1_{\mathrm{loc}}(\R^N)$ (instead of $\frac{1}{|x|^p}$) so that the following inequality holds
  \begin{equation} \label{WHS}
  \int_{\R^N} |w| |u|^p \dx \leq \displaystyle C \int_{\R^N} |\nabla u|^p\, \dx, \ \forall u \in \mathcal{D}^{1,p}(\R^N) 
  \end{equation}
 for some $C>0$. We denote the space of all admissible weight functions by $$\mathcal{H}_p(\R^N)=\{w\in L^1_{\mathrm{loc}}(\R^N): w \ \mbox{satisfies} \ \eqref{WHS}
 \}.$$ 
  One can use the Sobolev embedding to show that $L^{\frac{N}{p}}(\R^N) \subset \mathcal{H}_p(\R^N)$  \cite{Allegretto1995eigenvalues}.
 \tk{Further, using the Lorentz-Sobolev embedding, it can be shown that $L^{\frac{N}{p},\infty}(\R^N)\subset \mathcal{H}_p(\R^N)$  \cite[for $p=2$]{Visciglia}}.  Indeed, $L^{\frac{N}{p},\infty}(\R^N)$ does not exhaust $\mathcal{H}_p(\R^N)$, for instance, see \cite{Anoop2015weighted}. Further, we refer to \cite{Anoop2021compactness,anoop-p} for more nontrivial spaces contained in $\mathcal{H}_p(\R^N)$. {We refer to \cite{ABD} where authors provided various admissible spaces for a variant of the inequality \eqref{WHS} with different homogeneity on the left and right-hand side of the inequality. Extension of these results in Orlicz settings is recently obtained in \cite{ADR}.} 
 
 In this context, Maz'ya \cite[Theorem 8.5]{Mazya20002} gave a very intrinsic characterization of $\mathcal{H}_p(\R^N)$ using the $p$-capacity.
 \tk{Recall that the $p$-capacity of a compact set $F$ is defined as} 
  \[{\text{Cap}_p(F)}  = \inf \left\{ \displaystyle \int_{\Omega} | \nabla u |^p \dx: u \in  \mathcal{N}_p (F) \right \},\]
  where $ \mathcal{N}_p (F)= \{ u \in C^{\infty}_c(\R^N): u \geq 1 \ \mbox{on}\; F \}$. Maz'ya's characterization ensures that $w \in \mathcal{H}_p(\R^N)$ if and only if
   \begin{eqnarray*}
  \norm{w}_{\mathcal{H}_p}:= \sup\left\{ \frac{\int_{F} |w|\dx}{\text{Cap}_p(F)}: F \subset \R^N \ \mbox{is compact}; |F|\ne 0 \right\}<\infty.          \end{eqnarray*}
In this view, $\mathcal{H}_p(\R^N)$ is identified as $\mathcal{H}_p(\R^N)=\{w\in L^1_{\mathrm{loc}}(\R^N): \norm{w}_{\mathcal{H}_p}<\infty \}$. Indeed,  $ \norm{\cdot}_{\mathcal{H}_p}$ is a Banach function space norm on $ \mathcal{H}_p(\R^N)$ \cite{Anoop2021compactness}.  Next, one may look for $w \in \mathcal{H}_p(\R^N)$ for which the best constant in \eqref{WHS} is attained in $\mathcal{D}^{1,p}(\R^N)$. 
 Let $\mathcal{B}_{p}(w)$ be the best constant in \eqref{WHS} i.e., $\mathcal{B}_{p}(w)$ is the least possible constant so that \eqref{WHS} holds. Therefore, for $w \in \mathcal{H}_p(\R^N)$, we have
 \begin{equation} \label{bestHardy}
\mathcal{B}_{p}(w)^{-1}=\inf \left\{\int_{{\R^N}} |\nabla u|^p \dx : u \in \mathcal{D}^{1,p}({\R^N}), \int_{\R^N} |w| |u|^p \dx=1 \right\} \,.
\end{equation} 
Thus the best constant $\mathcal{B}_{p}(w)$ is attained in $\mathcal{D}^{1,p}({\R^N})$ if and only if \eqref{bestHardy} admits a minimizer.
 One of the simplest conditions that guaranties the existence of a minimizer for \eqref{bestHardy} is the compactness of the map  $$W(u)= \int_{\R^N} |w||u|^p  \dx $$ on $\mathcal{D}^{1,p}(\R^N)$ (i.e., for $u_n \wra u$ in $\mathcal{D}^{1,p}(\R^N)$, $W(u_n) \rightarrow W(u)$ \tk{in $\R$} as $n \ra \infty$). Many authors have given various sufficient conditions for the compactness of the map $W$. For example, Visciglia \cite{Visciglia} proved the compactness of $W$ for $w\in L^{\frac{N}{p},d}(\R^N)$ with $d<\infty$, which is later extended for $w\in \overline{{C}_c^{\infty}(\R^N)}$ in $L^{\frac{N}{p},\infty}(\R^N)$ \cite{anoop-p}. Furthermore, in \cite{Anoop2021compactness}, authors have identified the optimal space for the compactness of $W$, which is precisely $ \overline{{C}_c^{\infty}(\R^N)}$ in $\mathcal{H}_p(\R^N).$

 In this article, we are interested in the non-local {analogue} of \eqref{WHS}, namely, the fractional Hardy-{type} inequality:
\begin{align} \label{Fractional Hardy}
  \int_{\R^N} |w(x)||u(x)|^p  \dx \leq C  \int_{{\R^N \times \R^N}} \frac{|u(x)-u(y)|^p}{|x-y|^{N+sp}} \dxy:= { \norm{u}_{s,p}^p}\,, \ \forall u \in \mathcal{D}^{s,p}(\R^N), 
\end{align} where $0<s<1<p<\frac{N}{s}$, 
and $\D^{s,p}(\R^N)$ is the completion of $C^{\infty}_c(\R^N)$ with respect to the {norm} $\|\cdot\|_{s,p}$.
\begin{definition}[$(s,p)$-Hardy Potential]
A function $w \in L^1_{\mathrm{loc}}(\R^N)$ is called a $(s,p)$-Hardy potential if $w$ satisfies \eqref{Fractional Hardy}. We denote the space of $(s,p)$-Hardy potentials by $\mathcal{H}_{s,p}(\R^N)$.
\end{definition}
We know that the homogeneous weight function $w(x)=\displaystyle\frac{1}{|x|^{sp}}$, belongs to $\mathcal{H}_{s,p}(\R^N)$, see \cite{RS2008}. 
As in the local case, due to the fractional Sobolev inequality \cite[Theorem 6.5]{DNPV2012}, we have $L^r(\R^N) \subset \mathcal{H}_{s,p}(\R^N)$ for $r = \frac{N}{sp}$.
In fact, similar to the local case (i.e., $s=1$), we can also characterize the space $\mathcal{H}_{s,p}(\R^N)$  using the $(s, p)$-capacities, which is defined as follows:
\begin{definition}[$(s,p)$-Capacity]\rm
 For a \tk{compact set} $F $,
 we define
\begin{align*}
    {\rm{Cap}}_{s,p}(F)=\inf \big\{\norm{u}_{s,p}^p: u \in \mathcal{N}_{s,p}(F)\big\} \,,
\end{align*}
where $\mathcal{N}_{s,p}(F):= \{ u \in C^{\infty}_c(\R^N): u \geq 1 \text{ on } F\}.$ 
One may assume that $u =1$ on $F$ and $0 \leq u \leq 1$ in $\R^N$ (see \cite[Theorem 2.1]{Xiao}). 
\end{definition}

Motivated by the local case (i.e., $s=1$), for $w \in L^1_{\mathrm{loc}}(\R^N)$, we define
\begin{align}
 \norm{w}_{\mathcal{H}_{s,p}} =  \sup \left\{ \frac{\int_{F} |w(x)|\, \dx}{\mbox{Cap}_{s,p}(F)}: \tk{F \ \mbox{is compact}, |F|\neq 0} \right \}.
\end{align} Observe that, if $w$ satisfies \eqref{Fractional Hardy}, then for any \tk{compact set} $F$ and $u \in \mathcal{N}_{s,p}(F)$, we have
$$\int_{F} |w(x)| \dx \leq \int_{\R^N} |w(x)| |u(x)|^p \dx  \leq C \norm{u}_{s,p}^p.$$
This implies $\int_{F} |w(x)| \dx   \leq C \mathrm{Cap}_{s,p}(F).$ Therefore, $w$ necessarily satisfies $\norm{w}_{\mathcal{H}_{s,p}} < \infty$.
In fact, this condition is also sufficient for $w$ to satisfy \eqref{Fractional Hardy}, see Theorem \ref{H1}. 
Therefore, the space of $(s,p)$-Hardy potentials can be identified as
$$\mathcal{H}_{s,p}(\R^N)=\big\{w \in L^1_{\mathrm{loc}}(\R^N): \norm{w}_{\mathcal{H}_{s,p}} < \infty \big\} \,.$$
Indeed, $\|\cdot\|_{\mathcal{H}_{s,p}}$ is a Banach function norm on  $\mathcal{H}_{s,p}(\R^N)$ (for more details we refer to \cite[Section 30, Chapter 6]{Zaanen1958}). {For recent generalizations of Maz'ya-type characterization in other settings, we refer to \cite{DP, DKP, Hou2024}.}

{Now}, let $\mathcal{B}_{s,p}(w)$ be the best constant in \eqref{Fractional Hardy} i.e., $\mathcal{B}_{s,p}(w)$ is the least possible constant so that \eqref{Fractional Hardy} holds. Therefore, for $w \in \mathcal{H}_{s,p}(\R^N)$, we have
 \begin{equation} \label{bestHardy1}
\mathcal{B}_{s,p}(w)^{-1}=\inf \left\{ \|u\|_{{s,p}}^p : u \in \mathcal{D}^{s,p}({\R^N}), \int_{\R^N} |w| |u|^p \dx=1 \right\}.
\end{equation}
Similar to the local case, the compactness of the map $W$ on $\mathcal{D}^{s,p}({\R^N})$ ensures that the best constant $\mathcal{B}_{s,p}(w)$ is attained in $\mathcal{D}^{s,p}({\R^N})$.
\tk{For $0<s<1<p<N/s$, $W$ is compact on $\mathcal{D}^{s,p}({\R^N})$ if $w \in L^{\frac{N}{sp}}(\R^N)$ \cite{Pezzo2020}. For similar compactness results of $W$ on bounded domains, we refer to \cite{Iannizzotto2021monotonicity}.}
Let us now define the following closed subspace of $\mathcal{H}_{s,p}(\R^N)$:
$$\mathcal{H}_{s,p,0}(\R^N)=\overline{C_c(\R^N)} \ \mbox{in} \ \mathcal{H}_{s,p}(\R^N).$$
For $w \in \mathcal{H}_{s,p}(\R^N)$ and $x \in \R^N$, we define
\begin{align}\label{cw}
    \mathcal{C}_w(x):= \lim_{r \to 0} \|w \chi_{B_r(x)}\|_{\mathcal{H}_{s,p}},\ \mathcal{C}_w(\infty):=\lim_{r \to \infty} \|w \chi_{B_r(0)^c} \|_{\mathcal{H}_{s,p}} \text{ and } \mathcal{C}_w^*:=\sup_{x\in \R^N} \mathcal{C}_w(x) \,,
\end{align}
where $B_r(x)$ be the open ball of radius $r$ centered at $x$  and $\chi_A$ denotes the characteristic function of a subset $A$, i.e., $\chi_A(x)=1$ if $x \in A$ and it vanishes otherwise.
Our first main result proves the following equivalent characterizations for the compactness of $W$ on $\mathcal{D}^{s,p}(\mathbb{R}^N)$.
 \begin{theorem}
 \label{allinone1}
  Let {$0<s<1<p<\frac{N}{s}$ and } $w \in \mathcal{H}_{s,p}(\R^N)$. Then, the following statements are equivalent:
  \begin{enumerate}
   \item[(i)] The map $W:\mathcal{D}^{s,p}(\mathbb{R}^N) \to \R $, defined as $W(u)=\int_{\R^N} |w| |u|^p\, \dx$, is compact,
   \item[(ii)] $w$ has absolute continuous norm in $\mathcal{H}_{s,p}(\R^N)$, i.e., for any sequence of open sets $G_{n+1} \subset G_{n}$ for $n=1,2,\cdots$ and  $\displaystyle\bigcap_{n=1}^{\infty}G_n =\emptyset$, the norms $\|w\chi_{G_n}\|_{\mathcal{H}_{s,p}} \to 0$ as $n \to \infty$,
   \item[(iii)] $w \in \mathcal{H}_{s,p,0}(\R^N)$,
   \item[(iv)] $\C^*_w =0= \C_w(\infty)$.
  \end{enumerate}
  \end{theorem}
  
  Next, we are interested in studying the following fractional $p$-Laplace weighted eigenvalue problem: \begin{equation} \label{Weighetd eigenvalue problem}
    (-\Delta_{p})^{s}u = \lambda w(x) |u|^{p-2}u ~~\text{in}~\mathbb{R}^{N},
\end{equation}  where $0<s<1<p<\frac{N}{s}$ and $(-\Delta_{p})^{s}$ is the fractional $p$-Laplace operator defined on smooth functions as
\begin{align*}
    (-\Delta_{p})^{s}u(x) = 2 \lim_{\epsilon \rightarrow 0^{+}} \int_{\mathbb{R}^{N} \backslash B_{\epsilon}(x)} \frac{|u(x) - u(y)|^{p-2}(u(x)-u(y))}{|x-y|^{N+sp}}\,\mathrm{d}y\,\quad\text{for}~x \in \mathbb{R}^{N} \,,
\end{align*}
and the weight function $w = w_{1} - w_{2}~\text{with}~ w_{1},w_{2} \geq 0,$ is such that  $ w_{1} \in \mathcal{H}_{s,p,0}(\R^N)$ and $w_{2} \in L^{1}_{\mathrm{loc}}(\R^N)$.
If the weighted eigenvalue problem \eqref{Weighetd eigenvalue problem} has a non-trivial solution for some $\lambda \in \R$ 
i.e., there exists  $u \in \D^{s,p}(\mathbb{R}^{N}) \backslash \{0\}$ such that the following Euler-Lagrange equation 
\begin{equation} \label{Lagrange equation}
     \int_{\mathbb{R}^{N} \times \mathbb{R}^{N}} \dfrac{|u(x)-u(y)|^{p-2} (u(x)-u(y)) (v(x) -v(y))}{ |x-y|^{N+sp}} \dxy = \lambda \int_{\mathbb{R}^{N}} w|u|^{p-2}uv\, \dx,
\end{equation}
holds for all $v \in \D^{s,p}(\mathbb{R}^{N})$, then the scalar $\lambda \in \R$ is known to be the eigenvalue of \eqref{Weighetd eigenvalue problem}. The function $u (\neq 0)$ satisfying \eqref{Lagrange equation} is known as the eigenfunction corresponding to the eigenvalue $\lambda$. The first eigenvalue is the least possible eigenvalue defined by $\lambda_{1}:= \inf\{\|u\|_{s,p}^{p}: u \in \D^{s,p}(\R^N),\   \int_{\mathbb{R}^{N}} w|u|^{p}\,\dx =1\}$ and the corresponding eigenfunction is known as the first eigenfunction. An eigenvalue $\lambda$ is called principal if at least one of the eigenfunctions associated with it is of constant sign. If the eigenfunctions associated with the eigenvalue $\lambda$ are unique up to some constant multiple, then $\lambda$ is known as a simple eigenvalue.
If we consider the weighted eigenvalue problem \eqref{Weighetd eigenvalue problem} in a bounded domain $\Omega \subset \R^N$ with homogeneous Dirichlet boundary conditions i.e.,
\begin{align}\label{Weighetd eigenvalue problem bounded domain}
    \begin{cases}
    &(-\Delta_{p})^{s}u = \lambda w(x) |u|^{p-2}u \quad\text{in }\Omega,\\
    &u = 0 \quad\text{in }\mathbb{R}^N \setminus \Omega,
    \end{cases}
\end{align} 
then the existence, simplicity, and the principal {nature} of eigenvalues of \eqref{Weighetd eigenvalue problem bounded domain} have been discussed extensively in the literature. 
For $p=2, sp<N$ and $w = 1$, Servadei and Valdinoci \cite{Servadei2013} proved the simplicity and principality of the first eigenvalue $\lambda_1$ and established the existence of infinitely many eigenvalues to the problem \eqref{Weighetd eigenvalue problem bounded domain}, i.e., $ 0< \lambda_{1} < \lambda_{2} \leq \lambda_{3} \leq ... \leq \lambda_{k} \leq \cdots, ~\lambda_{k} \rightarrow \infty \text{ as }k \rightarrow \infty$.
For general $p\in (1,\infty)$ and $w=1$, the eigenvalue problem \eqref{Weighetd eigenvalue problem bounded domain} was first studied by Lindgren and Lindqvist \cite{Lindgren2014} (for the case $p \geq 2$) and later by Franzina and Palatucci \cite{Franzina2013} (for any
$p>1$).
For non-constant $w$, Pucci and Saldi \cite{Pucci2015eigenvalue} obtained the existence of a positive first eigenvalue of \eqref{Weighetd eigenvalue problem bounded domain} when $w \in L^{\alpha}(\Omega)$ is positive with $\alpha > \frac{N}{sp}$ and we refer to \cite{Iannizzotto2021monotonicity} for $\alpha = \frac{N}{sp}$. The author in \cite{Iannizzotto2021monotonicity} proved the existence of infinitely many eigenvalues with the first eigenvalue being simple, isolated and principal.
We also refer the articles \cite{Ho-Kanishka-Sim-Squassina-2017,Ho-Sim-2019} where more general weights were considered.
For the local case (i.e., $s=1$),  { we refer to \cite{anoop-p} and the references therein. The author in \cite{anoop-p} studied the problem \eqref{Weighetd eigenvalue problem bounded domain} for a general domain $\Omega$ (bounded or unbounded)} and showed the existence, simplicity, and uniqueness of the first eigenvalue of \eqref{Weighetd eigenvalue problem bounded domain}. Moreover, he obtained the existence of a sequence of infinite eigenvalues.  {For the fractional case, in contrast to the weighted eigenvalue problem \eqref{Weighetd eigenvalue problem bounded domain} which is posed on bounded domain, the literature dealing with problem \eqref{Weighetd eigenvalue problem} i.e., when $\Omega =\R^N$ in \eqref{Weighetd eigenvalue problem bounded domain}, is not very rich. In this direction,}
Pezzo and Quaas \cite[Theorem 1.1, Theorem 1.2]{Pezzo2020} considered the problem \eqref{Weighetd eigenvalue problem bounded domain} in two different cases $sp<N$ and $sp\geq N$. For $sp<N$, they assumed a sign changing $w \in L^{\frac{N}{sp}}(\mathbb{R}^{N})\cap L^{\infty}(\mathbb{R}^{N})$ with $w_1 \not\equiv 0$. On the other hand for $sp \geq N$, they considered $w \in L^{\infty}(\R^N)$, $w = w_1 -w_2$ such that: $(a) \ w_1(x) \geq 0 \text{ a.e. in } \R^N, w_1 \in L^{\frac{N}{sp}}(\mathbb{R}^{N})\cap L^{\infty}(\mathbb{R}^{N})$ and $(b) \ w_2(x) \geq \epsilon>0 \ \text{a.e. in }\R^N.$ In both the cases, the authors obtained the existence of infinite eigenvalues with the first eigenvalue being simple and principal. For $sp<N$, Cui and Sun \cite{Cui} recently obtained  { similar results} for eigenvalues as in \cite[Theorem 1.1]{Pezzo2020} by considering $w=w_1 - w_2, \ w_1,w_2 \geq 0$, such that $w_1\in L^{\frac{N}{sp}}(\mathbb{R}^{N}) \cap L^{\infty}(\mathbb{R}^{N}),w_2 \in L^{\infty}(\mathbb{R}^{N})$ and $w_1 \not\equiv 0$.
In our article, we generalize these result considering $w=w_1-w_2$ with  $ 0 \leq w_{1} \in \mathcal{H}_{s,p,0}(\R^N)$ and $0\leq w_{2} \in L^{1}_{\mathrm{loc}}(\R^N)$.
\begin{theorem} \label{Infinite eigenvalue}
 {Let $0<s<1<p<\frac{N}{s}$.} Assume that $w_{1} \in \mathcal{H}_{s,p,0}(\mathbb{R}^{N})$ and $w_{2} \in L^{1}_{\mathrm{loc}}(\R^N)$ with $w_{1} \not\equiv 0$, then there exists a sequence of eigenvalues $\{ \lambda_{k} \}$ for the problem (\ref{Weighetd eigenvalue problem}) such that $$ 0< \lambda_{1} < \lambda_{2} \leq \lambda_{3} \leq ... \leq \lambda_{k} \leq \cdots, \quad \lambda_{k} \rightarrow \infty \quad\text{as} ~~k \rightarrow \infty.$$ The first eigenvalue $\lambda_{1}$ is simple and principal. 
\end{theorem}
The rest of the article in organized as follows. In Section \ref{Prelims}, we recall some preliminary notions and known results that are required for the development of this article. Section \ref{Compactness} is dedicated for the proof of Theorem \ref{allinone1}. We discuss the weighted eigenvalue problem in Section \ref{WEVP}.

%% file: Prelim_new.tex
\section{Preliminaries}\label{Prelims}
In this section, we recall the notion of symmetrization, define Lorentz space, and provide some known results that will be used in the subsequent sections.

\subsection{Symmetrization}
The set of all extended real-valued Lebesgue measurable functions that are finite a.e. in $\mathbb{R}^{N}$, is denoted by $\mathcal{L}(\mathbb{R}^{N})$. For $f \in \mathcal{L}(\mathbb{R}^{N})$ and for $s>0$, we define $T_{f}(s) = \{ x:|f(x)|>s\} $ and the distribution function $\delta_{f}$ of $f$ is defined as
$$ \delta_{f}(s) = |T_{f}(s)|,~~\text{for}~s>0,$$
where $|E|$ denotes the $N$-dimensional Lebesgue measure of the set $E$. The one-dimensional decreasing rearrangement $f^{*}$ of $f$ is defined as below:
$$ f^{*}(t) = \begin{cases}
     \mathrm{ess\ sup} f,~\tk{\text{ if }}~t=0 \\
      \mathrm{inf}\{s>0:\delta_{f}(s)<t\},~\tk{\text{ if }} t>0 \,.
    \end{cases} $$
The map $f \mapsto f^{*}$ is not sub-additive. However, we obtain a sub-additive function from $f^{*}$, namely the maximal function $f^{**}$ of $f^{*}$, defined by 
$$ f^{**}(t) = \frac{1}{t}\int_{0}^{t}f^{*}(s) \,\mathrm{d}s,~~~t>0.$$
The sub-additivity of $f^{**}$ with respect to $f$ helps us to define norms in certain function spaces.

The Schwarz symmetrization of $f$ is defined by 
$$ f^{\star}(x) = f^{*}(\omega_{N}|x|^{N}),~~\forall~x \in \mathbb{R}^{N},$$
where $\omega_{N}$ is the measure of the unit ball in $\mathbb{R}^{N}$.
Next, we state an important inequality concerning the Schwarz symmetrization; see \cite[Theorem 3.2.10]{Edmunds}.
\begin{proposition}[Hardy-Littlewood inequality] Let  $f,g \in \mathcal{L}(\mathbb{R}^{N})$ be non-negative functions. Then
\begin{equation}\label{Hardy Littlewood}
    \int_{\mathbb{R}^{N}} f(x) g(x) \dx \leq \int_{\mathbb{R}^{N}} f^{\star}(x) g^{\star}(x) \dx = \tk{\int_{0}^{\infty} f^{*}(t) g^{*}(t) \dt}.
\end{equation}
\end{proposition}


\subsection{Lorentz spaces}
The Lorentz spaces are refinements of the usual Lebesgue spaces and introduced by Lorentz in 1950.
We refer to the book \cite{Edmunds} for more details on Lorentz spaces and related results.\\
Let $ (p,q) \in [1,\infty) \times [1,\infty]$, we define the Lorentz space $L^{p,q}(\R^N)$ as follow:
\begin{align*}
     L^{p,q}(\R^N) := \{ f \in \mathcal{L}(\R^N) ~:~ |f|_{(p,q)} < \infty \},
\end{align*}
where $|f|_{(p,q)}$ is a complete quasi-norm on $L^{p,q}(\R^N)$ and it is given by
\begin{align*}
    |f|_{(p,q)}:= \norm[\bigg]{t^{\frac{1}{p} - \frac{1}{q}} f^{*}(t)}_{L^{q}(0,\infty)} =  \begin{cases} 
       \left(\int_{0}^{\infty} \left[ t^{\frac{1}{p} - \frac{1}{q}} f^{*}(t) \right]^{q} \mathrm{d}t\right)^{\frac{1}{q}}, &\text{ \tk{if} } 1\leq q <\infty, \\
      \sup_{t>0}t^{\frac{1}{p}} f^{*}(t), &\text{\tk{if} } q= \infty.
   \end{cases}
\end{align*}
Moreover, if we define
\begin{align*}
    \|f\|_{(p,q)}:= \left\|t^{\frac{1}{p} - \frac{1}{q}} f^{**}(t)\right\|_{L^{q}(0,\infty)},
\end{align*}
then $\|f\|_{(p,q)}$ is a norm on $L^{p,q}(\R^N)$ and it is equivalent to the quasi-norm $ |f|_{(p,q)}$ {for $p \in (1,\infty]$ and $q \in [1,\infty]$} (see Lemma 3.4.6 of \cite{Edmunds}).   

\subsection{Br\'ezis-Lieb lemma and the discrete Picone-type
identity}
 
 The following lemma is due to Br\'ezis and Lieb  \cite[Theorem 1]{Brezis_Lieb_1983}.
\begin{lem}[Br\'ezis-Lieb lemma]
Let $(\Omega, \mathcal{A},\mu)$ be a measure space and $\{f_{n} \}$ be a sequence of measurable functions which are uniformly bounded in $L^{p}(\Omega,\mu)$ for some $0<p<\infty$. Moreover, if $\{ f_{n} \}$ converges to $f$ a.e., then
\begin{align*}
     \lim\limits_{n \ra \infty} \left| \|f_{n}\|_{  {L^{p}(\Omega,\mu)}} - \|f_{n}-f\|_{  {L^{p}(\Omega,\mu)}} \right| = \|f\|_{  {L^{p}(\Omega,\mu)}}.
\end{align*}
\end{lem}

Next, we recall a discrete Picone-type
identity in \cite[Lemma 6.2]{Amghibech2008}.
\begin{lem} \label{Picone identity}
Let $p \in (1,\infty)$ and \tk{$u,v: \mathbb{R}^{N} \rightarrow \mathbb{R}$  be two measurable functions. Then, for $u \geq 0 ~\text{and}~v>0$, we have $K(u,v) \geq 0$ in $\mathbb{R}^{N} \times \mathbb{R}^{N}$,} 
where
\begin{align}\label{def:K}K(u,v)(x,y) = |u(x)-u(y)|^{p} - |v(x)-v(y)|^{p-2}(v(x)-v(y)) \left(\frac{u(x)^p}{v(x)^{p-1}} - \frac{u(y)^p}{v(y)^{p-1}}\right).\end{align}
The equality holds if and only if $u = Cv$ a.e. for some constant $C$.
\end{lem}

\subsection{Some important estimates} \label{Imp_est}
  {Here} we recall the scaling property and the decay estimate of the nonlocal $(s,p)$-gradient given by Bonder et al. \cite{Bonder}. For $u \in  \mathcal{D}^{s,p}(\R^N)$, define
\begin{align*}
|D^s u(x)|^p = \int_{\R^N} \frac{|u(x+h)-u(x)|^p}{|h|^{N+sp}} \ \mathrm{d}h, \tk{\text{ for a.e.  } x\in \R^N.}
\end{align*}
\begin{lem}[{{\cite[Lemma 2.1]{Bonder}}}]\label{Bonder lemma one}
Let $\phi \in \mathcal{D}^{s,p}(\mathbb{R}^N)$. Given $r > 0$ and $x_0 \in \mathbb{R}^N$, define the scaled function
$\phi_{x_0,r}(x)=\phi(\frac{x-x_0}{r})$. Then 
\begin{align*}
    |D^s \phi_{x_0,r}(x)|^{p} = \frac{1}{r^{sp}} \left|D^s \phi\left(\frac{x-x_0}{r}\right)\right|^{p}.
\end{align*}
\end{lem}

\begin{lem}[{{\cite[Lemma 2.2]{Bonder}}}] \label{Bonder lemma two}
Let $\phi \in W^{1,\infty}(\R^N)$ be such that $\mbox{supp}(\phi) \subset B_{1}(0)$. Then, there exists a constant $\text{C}>0$ depends on $N,s,p$ and $\|\phi\|_{W^{1,\infty}}$ such that
\begin{align*}
    |D^s \phi(x)|^p \leq \text{C} \min\{1,|x|^{-(N+sp)}\} \,.
\end{align*}
\end{lem}

\begin{remark} \label{infestimate} \rm 
	Let $\phi \in W^{1,\infty}(\R^N)$ with compact support. \tk{Then, by Lemma \ref{Bonder lemma two}, $D^s \phi \in L^{\infty}(\R^N) \cap L^p(\R^N)$.}
	Now, let $\psi \in C_b^{1}(\R^N)$ be such that $0\leq \psi\leq 1$, $\psi=0$ on $B_1(0)$, and $\psi =1$ on $B_2(0)^c$. Then, $\phi := 1-\psi \in W^{1,\infty}(\R^N)$ with support in $B_2(0)$ and $|D^s \psi| = |D^s \phi|$. Thus,
    \begin{align*}
        |D^s \psi(x)|^p \leq C \min\{1,|x|^{-(N+sp)}\} \,,
    \end{align*}
	where $C>0$ depends on $N,s,p$ and $\|\psi\|_{W^{1,\infty}}$.
\end{remark}

%% file: Results_new.tex
\section{Compactness of the energy functional} \label{Compactness}
 {In this section, we prove our main result Theorem \ref{allinone1}. We divide the section into several subsections and develop the required tools for proving Theorem \ref{allinone1}. Indeed, each subsection contains certain results that are independently significant in their
own right.}
\subsection{\tk{Characterization} of the Hardy potentials}

\tk{First we recall the  Maz’ya-type characterization of fractional-Hardy-weights. This characterization has been established in different settings, e.g. see \cite{DP, DKP, Hou2024}.  For a proof of the following theorem, we refer to \cite[Proposition 3.1]{Dyda}.}
\begin{theorem} \label{H1}
Let  {$0<s<1<p<\frac{N}{s}$ and} $w \in L^{1}_{\mathrm{loc}}(\R^N)$. Then $ w \in \mathcal{H}_{s,p}(\mathbb{R}^{N})$ if and only if $\|w\|_{\mathcal{H}_{s,p}}<\infty$. Moreover, if $\|w\|_{\mathcal{H}_{s,p}}<\infty$ then  $$\int_{\R^N} w |u|^{p} \dx \leq C_{H} \|w\|_{\mathcal{H}_{s,p}} \int_{\mathbb{R}^{N} \times \mathbb{R}^{N}} \dfrac{|u(x)-u(y)|^{p}}{|x-y|^{N+sp}} \dxy, \ \ ~\forall ~ u \in \mathcal{D}_{}^{s,p}(\mathbb{R}^{N})$$
holds for some $C_H=C_H(N,s,p)>0$ which is independent of $w$.
\end{theorem}
For $0<s<1<p<\frac{N}{s}$, we define $(s,p)$-Capacity relative to a domain $\Omega$ in $\R^N$ as follows.
\begin{definition}[Relative $(s,p)$-Capacity]\rm
 Let $\Om \subset \R^N$ be a domain. For a \tk{compact set} $F \subset \Omega$,
 we define the capacity of $F$ with respect to $\Om$ by
\begin{align*}
    {\rm{Cap}}_{s,p}(F, \Omega)=\inf \big\{\norm{u}_{s,p}^p: \ u \in \mathcal{N}_{s,p}(F,\Omega)\big\} \,,
\end{align*}
where $\mathcal{N}_{s,p}(F, \Om):= \{ u \in C^{\infty}_c(\Om): u \geq 1 \text{ on } F\}.$ 
One may assume that $u =1$ on $F$ and $0 \leq u \leq 1$ in $\Om$ (see \cite[Theorem 2.1]{Xiao}).  
\end{definition}
The next proposition gives an interesting property of the $(s,p)$-capacity, which will play an important role in the proof of Theorem \ref{allinone1}.  The scaling property and the decay estimate of the nonlocal $(s,p)$-gradient from Section \ref{Imp_est} will be crucial for its proof. 
\begin{proposition} \label{propofcap}
   {Let $0<s<1<p<\frac{N}{s}$.} There exists $C_1, C_2 >0$ such that for a \tk{compact set} $F  \subset \R^N,$ 
   \begin{enumerate}[(i)]
    \item $ {\rm{Cap}}_{s,p}(F \cap \tk{\overline{B_r(x)}}, B_{2r}(x)) \leq C_1 {\rm{Cap}}_{s,p}(F \cap \tk{\overline{B_r(x)}}, \R^N),\ \forall r>0.$
    \item ${\rm{Cap}}_{s,p}(F \cap B_{2R}^c, \overline{B_R}^c) \leq C_2 {\rm{Cap}}_{s,p}(F \cap B_{2R}^c, \R^N), \ \forall R>0.$
   \end{enumerate}
\end{proposition}
\begin{proof}
$(i)$ Fix $x_0 \in \R^N$ and choose $\epsilon >0$ arbitrarily. Then, there exists $u \in C^{\infty}_c(\R^N)$ with $u \geq 1 $ on $F \cap \tk{\overline{B_r(x_0)}}$ such that 
$$\int_{\R^N \times \R^N} \frac{|u(x)-u(y)|^p}{|x-y|^{N+sp}} \dxy < \mbox{Cap}_{s,p}(F \cap \tk{\overline{B_r(x_0)}}, \R^N) + \epsilon \,.$$
Take $\phi \in C_c^{\infty}(\R^N)$ such that $0 \leq \phi \leq 1$, $\phi =1$ on $\overline{B_1(0)}$ and vanishes outside $B_2(0)$. Consider $v_r=u \phi_{x_0,r}$, where $\phi_{x_0,r}(y)=\phi(\frac{y-x_0}{r})$. Note that, $v_r \geq 1$ on $F \cap \tk{\overline{B_r(x_0)}}$. Now
\begin{align*}
 & \, \int_{\R^N \times \R^N} \frac{|v_{r}(x)-v_{r}(y)|^p}{|x-y|^{N+sp}} \dxy \\
 & \leq \tk{2^{p-1}} \left( \int_{\R^N \times \R^N} \frac{|u(x)-u(y)|^p}{|x-y|^{N+sp}}\dxy + \int_{\R^N \times \R^N} |u(y)|^p \frac{|\phi_{x_0,r}(x)-\phi_{x_0,r}(y)|^p}{|x-y|^{N+sp}} \dxy \right)\\
 & := \tk{2^{p-1}}(I_1 + I_2) \,.
\end{align*}
We estimate $I_2$ as follows.
\begin{eqnarray*}
    I_2  & = & \int_{\R^N} |u(y)|^p \int_{\R^N} \frac{|\phi_{x_0,r}(x)-\phi_{x_0,r}(y)|^p}{|x-y|^{N+sp}} \ \dx  \dy \\
 & = & \int_{\R^N} |u(y)|^p \int_{\R^N} \frac{|\phi_{x_0,r}(y+z)-\phi_{x_0,r}(y)|^p}{|z|^{N+sp}} \ \mathrm{d}z \dy  \\
  & \tk{=} & \int_{\R^N} |u(y)|^p |D^s \phi_{x_0,r}(y)|^p   \dy \\
  & = & \frac{1}{r^{sp}} \int_{\R^N} |u(y)|^p \left|D^s \phi\left(\frac{y-x_0}{r}\right)\right|^p  \dy \ \ \mbox{  \cite[Lemma \ 2.1]{Bonder}}  \\
  & \leq & \|D^s \phi\|_{\infty}^p \left( \int_{\R^N} |u(y)|^{p_s^*}  \dy \right)^{\frac{N-sp}{N}} \left( \int_{\R^N} \left|\frac{D^s \phi(y)}{\|D^s \phi\|_{\infty}} \right|^{\frac{N}{s}}   \dy  \right)^{\frac{sp}{N}} \\
  & \leq & \|D^s \phi\|_{\infty}^p \left( \int_{\R^N} |u(y)|^{p_s^*} \dy \right)^{\frac{N-sp}{N}} \left( \int_{\R^N} \left|\frac{D^s \phi(y)}{\|D^s \phi\|_{\infty}} \right|^{p}   \dy  \right)^{\frac{sp}{N}} \\
  & \leq & C  \int_{\R^N \times \R^N} \frac{|u(x)-u(y)|^p}{|x-y|^{N+sp}}  \dx  \dy
  \ = \ C I_1 \ \ \ \mbox{[by Remark \ref{infestimate}]}\,.
\end{eqnarray*}
  Thus, we obtain
  \begin{align*}
      \mbox{Cap}_{s,p}(F \cap \tk{\overline{B_r(x_0)}},  B_{2r}(x_0)) & \leq \int_{\R^N \times \R^N} \frac{|v_{r}(x)-v_{r}(y)|^p}{|x-y|^{N+sp}}\dxy 
       \leq \tk{2^{p-1}}(1+C) I_1 \\ & <C_{1}\mbox{Cap}_{s,p}(F \cap \tk{\overline{B_r(x_0)}}, \R^N) +C_{1} \epsilon \,,
  \end{align*}
 where $C_{1} = \tk{2^{p-1}}(1+C).$ By taking $\epsilon \ra 0$, we prove $(i).$
  
  $(ii)$ Choose $\epsilon >0$ arbitrarily. Then there exists $u \in C^{\infty}_c(\R^N)$ with $u \geq 1 $ on $F \cap B_{2R}(0)^c$ such that 
$$\int_{\R^N \times \R^N} \frac{|u(x)-u(y)|^p}{|x-y|^{N+sp}} \dxy< \mbox{Cap}_{s,p}(F \cap B_{2R}(0)^c, \R^N) + \epsilon \,.$$
Take $\phi \in C_b^{\infty}(\R^N)$ such that $0 \leq \phi \leq 1$, $\phi =0$ on $B_1(0)$ and $\phi =1$ on $B_2(0)^c$. Consider $v_{R}=u \phi_{R}$, where $\phi_{R}(x)=\phi(\frac{x}{R})$. Now
\begin{align*}
 & \, \int_{\R^N \times \R^N} \frac{|v_{R}(x)-v_{R}(y)|^p}{|x-y|^{N+sp}} \dxy \\  &  \leq \tk{2^{p-1}} \left( \int_{\R^N \times \R^N} \frac{|u(x)-u(y)|^p}{|x-y|^{N+sp}} \dxy+ \int_{\R^N \times \R^N} |u(y)|^p \frac{|\phi_{R}(x)-\phi_{R}(y)|^p}{|x-y|^{N+sp}} \dxy \right) \\
 & := \tk{2^{p-1}}(I_1 + I_2) \,.
\end{align*}
We estimate $I_2$ as follows.
\begin{align*}
I_2  & = \int_{\R^N} |u(y)|^p \int_{\R^N} \frac{|\phi_{R}(x)-\phi_{R}(y)|^p}{|x-y|^{N+sp}} \dxy \\
 & = \int_{\R^N} |u(y)|^p \int_{\R^N} \frac{|\phi_{R}(y+z)-\phi_{R}(y)|^p}{|z|^{N+sp}} \ \mathrm{d}z \mathrm{d}y  \\
  & \tk{=} \int_{\R^N} |u(y)|^p |D^s \phi_{R}(y)|^p  \dy \\
  & = \frac{1}{R^{sp}} \int_{\R^N} |u(y)|^p \left|D^s \phi\left(\frac{y}{R}\right)\right|^p  \dy \ \ \  \mbox{\cite[Lemma \ 2.1]{Bonder}}  \\
  & \leq  \left( \int_{\R^N} |u(y)|^{p_s^*} \mathrm{d}y \right)^{\frac{N-sp}{N}} \left( \int_{\R^N} |D^s \phi(y)|^{\frac{N}{s}}   \mathrm{d}y  \right)^{\frac{sp}{N}} \\
  & \leq \|D^s \phi\|_{\infty}^p \left( \int_{\R^N} |u(y)|^{p_s^*}  \dy \right)^{\frac{N-sp}{N}} \left( \int_{\R^N} \left|\frac{D^s \phi(y)}{\|D^s \phi\|_{\infty}} \right|^{\frac{N}{s}}  \dy  \right)^{\frac{sp}{N}} \\
 & \leq \|D^s \phi\|_{\infty}^p \left( \int_{\R^N} |u(y)|^{p_s^*} \dy \right)^{\frac{N-sp}{N}} \left( \int_{\R^N} \left|\frac{D^s \phi(y)}{\|D^s \phi\|_{\infty}} \right|^{p}  \dy  \right)^{\frac{sp}{N}} \\ 
  & \leq C  \int_{\R^N \times \R^N} \frac{|u(x)-u(y)|^p}{|x-y|^{N+sp}} \dxy
  \ = \ C I_1 \ \ \mbox{[by Remark \ref{infestimate}]}.
  \end{align*}
  Therefore, the result follows.
\end{proof}
The following proposition provides a necessary and sufficient condition for the weights $w\in L^1_{\mathrm{loc}}(\mathbb{R}^N)$ to lie in the space $\mathcal{H}_{s,p,0}(\R^N)$.
\begin{proposition} \label{charF}
Let  {$0<s<1<p<\frac{N}{s}$ and} $w \in L^1_{\mathrm{loc}}(\mathbb{R}^N)$. Then, $w \in \mathcal{H}_{s,p,0}(\R^N)$ if and only if for every $\epsilon >0$, there exists $w_{\epsilon} \in L^{\infty}(\R^N) $ such that $|supp(w_{\epsilon})|< \infty$ (where $|E|$ denotes the $N$-dimensional Lebesgue measure of the set $E$) and $\norm{w-w_{\epsilon}}_{\mathcal{H}_{s,p}}< \epsilon.$ 
\end{proposition}

\begin{proof}
Let $w \in \mathcal{H}_{s,p,0}(\R^N)$ and $\epsilon >0$ be given. By definition of $\mathcal{H}_{s,p,0}(\R^N)$, there exists $w_{\epsilon} \in C_c(\R^N)$ such that $\norm{w-w_{\epsilon}}_{\mathcal{H}_{s,p}} < \epsilon .$ This $w_{\epsilon}$ \tk{fulfills} our requirements. 
For the converse part, take a $w$ satisfying the hypothesis. Let $\epsilon >0 $ be arbitrary. Then there exists $w_{\epsilon} \in L^{\infty}(\R^N)$ such that 
$|supp(w_{\epsilon})|< \infty$ and $\norm{w-w_{\epsilon}}_{\mathcal{H}_{s,p}}< \frac{\epsilon}{2}.$
Thus, $w_{\epsilon} \in L^{\frac{N}{sp}}(\R^N)$ and hence there exists $ \phi_{\epsilon} \in C_c^{}(\R^N)$ such that 
 $\norm{w_{\epsilon}-\phi_{\epsilon}}_{\frac{N}{sp}} < \frac{\epsilon}{2C}$,
 where $C$ is the embedding constant for the embedding $ L^{\frac{N}{sp}}(\R^N)$ into $\mathcal{H}_{s,p}(\R^N)$, \tk{see Proposition \ref{Lorentz space embedding} for a proof of this embedding}. Now by triangle inequality, we obtain $\norm{w-\phi_{\epsilon}}_{\mathcal{H}_{s,p}}< \epsilon$ as required.
\end{proof}


\subsection{Some important embeddings}
In this subsection, we will prove some important embedding results to help us reach our final goal later. First, we prove the following result:
 \begin{proposition} \label{Lorentz space embedding}
For  {$0<s<1<p<\frac{N}{s}$}, $L^{\frac{N}{sp}, \infty}(\mathbb{R}^{N})$ is continuously embedded in $\mathcal{H}_{s,p}(\mathbb{R}^{N})$.
\end{proposition}
\begin{proof}
Observe that $\mathrm{Cap}_{s,p}(F^{\star})  \leq \mathrm{Cap}_{s,p}(F)$. The inequality follows from P\'olya-Szeg\"o inequality  \cite[Theorem 9.2]{Almgren}. Also, we know that $\mathrm{Cap}_{s,p}(F^{\star}) \geq \mathcal{K}_{N,s,p} R^{N-sp},$ where $R$ is the radius of $F^{\star}$ and $ \mathcal{K}_{N,s,p} > 0 $ is a constant independent of $R$ \cite[Theorem 3]{Xiao}. Now for a \tk{compact} set $F$ with Lebesgue measure $|F|$, \[
\dfrac{\int_{F}|w|(x) \dx}{\mathrm{Cap}_{s,p}(F)} \leq \dfrac{\int_{F^{\star}} w^{\star}(x) \dx}{\mathrm{Cap}_{s,p}(F^{\star})} \leq \dfrac{\tk{\int_{0}^{|F|} w^{\ast}(t) \dt}}{\mathcal{K}_{N,s,p} R^{N-sp}} = \dfrac{\omega_{N} R^{N} w^{\ast \ast}(\omega_{N} R^{N})}{\mathcal{K}_{N,s,p} R^{N-sp}} = \dfrac{\tk{\omega_{N}}R^{sp}~ w^{\ast \ast}(\omega_{N} R^{N})}{\mathcal{K}_{N,s,p}},
\]
where we use Hardy-Littlewood inequality \cite[Thoerem 3.2.10]{Edmunds} in the first and second inequality. By setting $\omega_{N}{R}^{N} = t,$ we get
\begin{eqnarray*}
\dfrac{\int_{F}|w|(x) \dx}{\mathrm{Cap}_{s,p}(F)} \leq \tk{C_{N,s,p}} \|w\|_{(\frac{N}{sp}, \infty)}.
\end{eqnarray*}
Now take the supremum \tk{over compact sets $F \subset \R^N$} to obtain,
\begin{eqnarray*}
\|w\|_{\mathcal{H}_{s,p}} \leq \tk{C_{N,s,p}} \|w\|_{(\frac{N}{sp}, \infty)} \end{eqnarray*} 
with $\tk{C_{N,s,p}} > 0$ and depending on $N,s$ and $p$.
\end{proof}
Let us define the following spaces 
$$L^{\frac{N}{sp},\infty}_0(\mathbb{R}^{N}) = \overline{C_{c}^{}(\mathbb{R}^{N})} ~~\mbox{in} ~~L^{\frac{N}{sp},\infty}(\mathbb{R}^{N}),$$
$$\mathcal{H}_{s,p,0}(\mathbb{R}^{N}) =  \overline{C_{c}^{}(\mathbb{R}^{N})} ~~\mbox{in} ~~\mathcal{H}_{s,p}(\mathbb{R}^{N}).$$
\begin{proposition}
Let $0<s<1<p<\frac{N}{s}$. Then $L^{\frac{N}{sp},\infty}_0(\mathbb{R}^{N}) \subset \mathcal{H}_{s,p,0}(\mathbb{R}^{N})$.
\end{proposition}
\begin{proof}
\tk{Recall that} $L^{\frac{N}{sp},\infty}_0(\mathbb{R}^{N})$ is the closure of $C_{c}^{}(\mathbb{R}^{N})$ in $L^{\frac{N}{sp}, \infty}(\mathbb{R}^{N})$ and $\mathcal{H}_{s,p,0}(\mathbb{R}^{N})$ is the closure of $C_{c}^{}(\mathbb{R}^{N})$ in $\mathcal{H}_{s,p}(\mathbb{R}^{N})$. From Proposition \ref{Lorentz space embedding}, we have $\|\cdot\|_{\mathcal{H}_{s,p}} \leq C \|\cdot\|_{(\frac{N}{sp}, \infty)}$. Therefore, it is immediate that $L^{\frac{N}{sp},\infty}_0(\mathbb{R}^{N})$ is contained in $\mathcal{H}_{s,p,0}(\mathbb{R}^{N})$.
\end{proof}

In the following proposition, we establish the Lorentz-Sobolev embedding for $\D^{s,p}(\mathbb{R}^{N})$  {using the fractional Hardy-type inequality \eqref{Fractional Hardy}.}
\begin{proposition}
Let $0<s<1<p<\frac{N}{s}$, then $\D^{s,p}(\mathbb{R}^{N})$ is continuously embedded in the Lorentz space $L^{p^{*}_{s},p}(\mathbb{R}^{N})$, where $p^{*}_{s} = \frac{Np}{N-sp}.$
\end{proposition}
\begin{proof}
 Let $w \in \mathcal{H}_{s,p}(\mathbb{R}^{N})$ be such that $w^{\star} \in \mathcal{H}_{s,p}(\mathbb{R}^{N}).$ Then \tk{using Theorem \ref{H1}} and the P\'olya-Szeg\"o inequality \cite[Theorem 9.2]{Almgren}, we have
\begin{eqnarray*}
\int_{\mathbb{R}^{N}} w^{\star} |u^{\star}|^{p} \dx &\leq & C \|w^{\star}\|_{\mathcal{H}_{s,p}} \int_{\mathbb{R}^{N} \times \mathbb{R}^{N}} \dfrac{|u^{\star}(x) - u^{\star}(y)|^{p}}{|x-y|^{N+sp}} \dxy \\ 
 &\leq & C \|w^{\star}\|_{\mathcal{H}_{s,p}} \int_{\mathbb{R}^{N} \times \mathbb{R}^{N}} \dfrac{|u(x) - u(y)|^{p}}{|x-y|^{N+sp}} \dxy,\ \ \ \forall u \in \D^{s,p}(\mathbb{R}^{N}).
\end{eqnarray*}
In Particular, for $w(x) = \dfrac{1}{\omega_{N}^{\frac{sp}{N}} |x|^{sp}},~w^{*}(t) = \dfrac{1}{t^{\frac{sp}{N}}},~\mbox{and}~ \|w^{\star}\|_{\mathcal{H}_{s,p}} \leq C(N,p,s).$ Also we have $$\int_{\mathbb{R}^{N}} w^{\star} |u^{\star}|^{p} \dx = \int_{0}^{\infty} w^{*}(t) |u^{*}(t)|^{p} \mathrm{d}t.$$ Thus, from the above inequality we have, 
\begin{eqnarray*}
 \int_{0}^{\infty} \dfrac{1}{t^{\frac{sp}{N}}} |u^{*}(t)|^{p} \mathrm{d}t&=& \int_{0}^{\infty} w^{*}(t) |u^{*}(t)|^{p} \mathrm{d}t = \int_{\mathbb{R}^{N}} w^{\star} |u^{\star}|^{p} \dx\\
 &\leq& C(N,p,s) \int_{\mathbb{R}^{N} \times \mathbb{R}^{N}} \dfrac{|u(x) - u(y)|^{p}}{|x-y|^{N+sp}} \dxy,\ \ \  \forall u \in \D^{s,p}(\mathbb{R}^{N}).
\end{eqnarray*}
 The left-hand side of the above inequality is $|u|^{p}_{(p^{*}_{s},p)}$, a quasi-norm equivalent to the norm $\|u\|^{p}_{(p^{*}_{s},p)}$ in $L^{p^{*}_{s},p}(\mathbb{R}^{N})$. This completes the proof.
\end{proof}

\subsection{Concentration compactness}
Let $\mathbb{M} (\R^N)$ be the space of all regular, finite,  Borel-signed measures on $\R^N.$ Then $\mathbb{M} (\R^N)$ is a Banach space with respect to the norm $\norm{\mu}=|\mu|(\R^N)$ (total variation of the measure $\mu$). By Riesz representation theorem, we know that $\mathbb{M}(\R^N)$ is the dual of $C_0(\R^N)$ (= $\overline{C_c(\R^N)}$ in $L^{\infty}(\R^N)$) \cite[Theorem 14.14, Chapter 14]{Border}. The next proposition follows from the uniqueness part of the Riesz representation theorem.
\begin{proposition} \label{defmeasure}
 Let $\mu \in \mathbb{M}(\R^N)$ be a positive measure. Then for an open $V \subseteq \R^N$,
 \[ \mu(V)= \sup \left \{ \int_{\R^N} \phi \, \dm : 0 \leq \phi \leq 1, \phi \in C_c^{\infty}(\R^N) \ with \ supp(\phi) \subseteq V   \right \}\]
and for any Borel set $E \subseteq \R^N$, $\mu(E):= \inf \big\{ \mu(V) : E \subseteq V \ \mbox{and} \ V \text {is open } \big\}$.
\end{proposition}
A sequence $(\mu_n)$ is said to be \tk{weak*} convergent to $\mu$ in $\mathbb{M}(\R^N)$, if
 \begin{eqnarray*}
  \int_{\R^N} \phi \, \dm_n \ra  \int_{\R^N} \phi\, \dm, \ as \ n \ra \infty,  \forall\,  \phi \in C_0(\R^N).
 \end{eqnarray*}
 In this case, we denote $\mu_n \wrastar \mu$. The following proposition is a consequence of the Banach-Alaoglu theorem \cite[Chapter 5, Section 3]{Conway}, which states that for any normed linear space $X$, the closed unit ball in  $X^*$ is \tk{weak*} compact.
\begin{proposition} \label{BanachAlaoglu}
 Let $(\mu_n)$ be a bounded sequence in $\mathbb{M}(\R^N)$, then there exists $\mu \in \mathbb{M}(\R^N)$ such that $\mu_n \overset{\ast}{\rightharpoonup} \mu$ up to a subsequence.
\end{proposition}
\begin{proof}
 Recall that, if $X=C_0(\R^N)$ then by Riesz Representation theorem \cite[Theorem 14.14, Chapter 14]{Border} $X^*=\mathbb{M}(\R^N)$. Thus, the proof follows from the Banach-Alaoglu theorem \cite[Chapter 5, Section 3]{Conway}. 
\end{proof} 
For $u_n, u\in {\mathcal{D}^{s,p}(\R^N)}$, \tk{$w \in\mathcal{H}_{s,p}(\mathbb{R}^N)$} and a Borel set  $ E $ in   $\R^N,$ we  denote 
\begin{align*}
\nu_n(E)=\int_E w|u_n - u|^p \, \dx \,, &\,  \quad \Ga_n(E)=\int_E|D^s (u_n-u)|^p \, \dx \,  \\ 
\widetilde{\Ga}_n(E) &=\int_E|D^s u_n|^p \, \dx \,.
 \end{align*}
If $u_n\wra u$ in $\mathcal{D}^{s,p}(\R^N)$, then $\nu_n$, $\Ga_n$ and $\widetilde{\Ga}_n$ 
  have \tk{weak*} convergent sub-sequences (Proposition \ref{BanachAlaoglu}) in $\mathbb{M}(\R^N)$. Let
   \[\nu_n \overset{\ast}{\rightharpoonup} \nu \,, \qquad \Gamma_n \overset{\ast}{\rightharpoonup} \Gamma \,, \qquad \widetilde{\Gamma}_n \overset{\ast}{\rightharpoonup} \widetilde{\Gamma} \  \mbox{ in } \mathbb{M}(\R^N).\] 
We develop a $w$-depended concentration compactness lemma using our concentration function $\C_w$ (see \eqref{cw} for the definition). Our results are the non-local extensions of the results in \cite{Anoop2021compactness}.

The following Lemma is due to \cite[Remark 2.5]{Bonder}.
\begin{lemma} \label{lem: 1}
Let $0<s<1<p<\frac{N}{s}$ and $\phi \in W^{1,\infty}(\R^N)$ with compact support.
Let $u_n \wra u$ in $\mathcal{D}^{s,p}(\R^N)$. Then
\begin{align*}
    \displaystyle \lim_{n\ra \infty} \int_{\R^N} |u_n(x)-u(x)|^p |D^s \phi(x)|^p \ \dx = 0 \,.
\end{align*}
\end{lemma}
\begin{remark} \rm
Lemma \ref{lem: 1} also holds if we replace $\phi$ with  $\psi \in C_b^{\infty}(\R^N)$ with $0 \leq \psi \leq 1$, $\psi = 0$ on $B_1(0)$ and $\psi =1$ on $B_2(0)^c$. 
\end{remark}
\begin{corollary} \label{passing_limit}
Let  {$0<s<1<p<\frac{N}{s}$ and} $u_n \wra u$ in $\mathcal{D}^{s,p}(\R^N)$. Let $\phi \in W^{1,\infty}(\R^N)$ with compact support or $\phi \in C_b^{\infty}(\R^N)$ with $0 \leq \phi \leq 1$, $\phi = 0$ on $B_1(0)$ and $\phi =1$ on $B_2(0)^c$. Then, for $v_n=(u_n-u)\phi$, we have
\begin{align*}
 \tk{\overline{\lim_{n \ra \infty}}} \int_{\R^N \times \R^N} \frac{|v_n(x)-v_n(y)|^p}{|x-y|^{N+sp}} \dxy \leq  \tk{\overline{\lim_{n \ra \infty}}} \int_{\R^N \times \R^N} |\phi(y)|^p\frac{|(u_n-u)(x)-(u_n-u)(y)|^p}{|x-y|^{N+sp}} \dxy  \,, 
\end{align*}
\end{corollary} 
\noi Next, we prove the absolute continuity of $\nu$ with respect to $\Gamma$. 
 \begin{lemma}\label{mlc1}
 Let  {$0<s<1<p<\frac{N}{s}$. Assume that} $w \in \mathcal{H}_{s,p}(\R^N) $, $w \geq 0$ and $u_n \wra u$ in $ \mathcal{D}^{s,p}(\R^N)$. Then for any Borel set $E$ in $\R^N$,
 \begin{align*}
     \nu(E) \leq C_{H} \C^*_w \Gamma (E),
 \end{align*}
  where $\C^*_w = \displaystyle\sup_{x \in \R^N} \C_w(x)$ and $C_H>0$ is as in Theorem \ref{H1}.
  \end{lemma}
  \begin{proof}
As $u_n \wra u$ in $ \mathcal{D}^{s,p}(\R^N)$, $u_n \ra u$ in $L^{p}_{\mathrm{loc}}(\R^N)$.  For $\Phi \in C_c^{\infty}(\R^N)$, $(u_n-u) \Phi \in \D^{s,p}(\R^N)$. Thus, denoting $v_n=(u_n-u)\Phi$, we have
\begin{align*}
  \int_{\R^N} |\Phi|^p \ \mathrm{d} \nu_n  =  \int_{\R^N} w|(u_n-u)\Phi|^p \dx  \leq   C_H \norm{w}_{\mathcal{H}_{s,p}}  \int_{\R^N \times \R^N} \frac{|v_n(x)-v_n(y)|^p}{|x-y|^{N+sp}} \dxy  \,.
\end{align*}
Taking $n \ra \infty$ and using Corollary \ref{passing_limit}, we obtain
 \begin{align} \label{forrmk}
     \int_{\R^N} |\Phi|^p \ \mathrm{d} \nu \leq C_{H} \norm{w}_{\mathcal{H}_{s,p}} \int_{\R^N} |\Phi|^p \ \mathrm{d} \Gamma .
 \end{align}
Now by Proposition \ref{defmeasure}, we get 
  \begin{equation}\label{measureinequality1}
   \nu(E) \leq C_{H} \norm{w}_{\mathcal{H}_{s,p}} \Gamma (E),  \  \forall E \ \text{Borel in} \ \R^N.
  \end{equation}
 In particular, $\nu \ll \Gamma$ and hence by Radon-Nikodym theorem, 
 \begin{equation} \label{measureinequality}
  \nu(E) = \int_E  \frac{\mathrm{d} \nu}{\mathrm{d} \Gamma} \  \mathrm{d}\Gamma \ , \forall E \ \text{Borel  in} \ \R^N. 
 \end{equation}
 Further, by Lebesgue differentiation theorem (page 152-168 of \cite{Federer}), we have 
 \begin{equation} \label{Lebdiff}
  \frac{\mathrm{d} \nu}{\mathrm{d} \Gamma}(x) = \lim_{r \ra 0} \frac{\nu (B_r(x))}{\Gamma (B_r(x))}.
 \end{equation}
 Now replacing $w$ by $w \chi_{B_r(x)}$ and proceeding as before,
 \[ \nu(B_r(x)) \leq C_{H} \norm{w \chi_{B_r(x)}}_{\mathcal{H}_{s,p}} \ \Gamma (B_r(x)).\] 
 Thus from \eqref{Lebdiff} we get 
\begin{eqnarray} \label{21}
 \frac{\mathrm{d} \nu}{\mathrm{d} \Gamma} (x) \leq C_{H} \C_w(x)
\end{eqnarray} 
and hence 
$\norm{\frac{\mathrm{d} \nu}{\mathrm{d} \Gamma}}_{\infty} \leq C_{H} \C^*_w$. Now from \eqref{measureinequality} we obtain $\nu(E) \leq C_{H} \C^*_w \Gamma (E)$ 
for all Borel subsets $E$ of $\R^N$.
\end{proof}    
   
The next Lemma gives a lower estimate for the measure $\tilde{\Gamma}.$ Similar estimate is obtained in Lemma 2.1 of \cite{Smets}. \tk{Let $\sum_w := \{x \in \R^N : C_w(x)>0 \}$}. We will assume that $\overline{\sum_w}$ is of Lebesgue measure $0$.
\begin{lem} \label{lemma9}
Let  {$0<s<1<p<\frac{N}{s}$ and} $w \in \mathcal{H}_{s,p}(\R^N)$ be such that $w \ge 0$ and $|\overline{\sum_w}|=0$. If $u_n \wra u$ in $ \D^{s,p}(\R^N)$, then 
\begin{equation*}
   \tilde{\Gamma} \geq \begin{cases}
     |D^s u|^p + \frac{\nu}{C_{H} \C_w^*}, \quad \text{if} \ \C_w^* \neq 0, \\
     |D^s u|^p, \quad \text{otherwise} \,,
    \end{cases}
\end{equation*}
 where $C_H>0$ is as in Theorem \ref{H1}.
\end{lem}
\begin{proof}
Our proof splits into three steps.\\
{\bf Step 1:} $\tilde{\Ga} \geq |D^s u|^p.$ \tk{As $u_n \wra u$ in $ \mathcal{D}^{s,p}(\R^N)$, $u_n \ra u$ a.e. in $\R^N$. Using Fatous lemma, one can see that $|D^s u| \leq \liminf_{n \ra \infty}  |D^s u_n|$ a.e. in $\R^N$. Now,  let $\phi \in C_c^{\infty}(\R^N)$ with $0\leq \phi \leq 1$, we need to show that
$\int_{\R^N} \phi \ \mathrm{d}\tilde{\Ga} \geq \int_{\R^N} \phi |D^s u|^p \, \dx. $
Notice that
\begin{align*}\int_{\R^N} \phi \, \mathrm{d}\tilde{\Ga}= \lim_{n \ra \infty} \int_{\R^N} \phi \, \mathrm{d}\tilde{\Ga}_n  = \lim_{n \ra \infty} \int_{\R^N} \phi |D^s u_n|^p \, \dx \geq \int_{\R^N} \phi |D^s u|^p \, \dx ,\end{align*}
where the last inequality follows from Fatous lemma.}

\noi {\bf Step 2:} $\tilde{\Gamma}=\Gamma$, on $\overline{\sum_w}.$ Let $E\subset\overline{\sum_w}$ be a Borel set. Thus, for each $m \in \N$, there exists an open subset $O_{m}$ containing $E$ such that $|O_m|=|O_{m} \setminus E| < \frac{1}{m}$. Let $\var >0$ be given. Then, for any $\phi \in C_c^{\infty}(O_{m})$ with $0 \leq \phi \leq 1$, we have
\begin{align*}
    &\bigg|\int_{\R^N}  \phi \ \mathrm{d}\Gamma_n \, \dx  - \int_{\R^N}  \phi \ d\tilde{\Gamma}_n \, \dx\bigg| \\ 
    &= \left|\int_{\R^N}  \phi |D^s (u_n-u)|^p \, \dx -\int_{\R^N}  \phi |D^s u_n|^p \, \dx \right| \\
     &\leq  \int_{\R^N \times \R^N} \phi(x) \frac{\big| |u_n(x)-u_n(y)|^p - |(u_n-u)(x)-(u_n-u)(y)|^p \big|}{|x-y|^{N+sp}}\, \dxy \\
    &\leq  \epsilon \int_{\R^N \times \R^N} \phi(x) \frac{ |u_n(x)-u_n(y)|^p}{|x-y|^{N+sp}}\, \dxy  + c(\epsilon,p) \int_{\R^N \times \R^N} \phi(x) \frac{ |u(x)-u(y)|^p}{|x-y|^{N+sp}}\ \dxy \\
    & \leq \epsilon \int_{\R^N \times \R^N} \frac{ |u_n(x)-u_n(y)|^p}{|x-y|^{N+sp}}\, \dxy + c(\epsilon,p) \int_{O_m}  |D^s u|^p \, \dx \,.
\end{align*}
\tk{Now taking $n \rightarrow \infty$, we obtain $
 \left|\int_{\R^N}  \phi \, \mathrm{d}\Gamma  -\int_{\R^N}  \phi \, \mathrm{d}\tilde{\Gamma}  \right|  \leq  \epsilon L + c(\epsilon,p) \int_{O_{m}}   |D^s u|^p \ \dx
$, where $L=\sup_{n \in \N} \{\|u_n\|_{s,p}\} <\infty$.
Therefore,
\begin{eqnarray*} 
\left|\Gamma (O_m)-\tilde{\Gamma} (O_m) \right| \leq \epsilon L + c(\epsilon,p) \int_{O_{m}}   |D^s u|^p \, \dx \,.
\end{eqnarray*}
Thus, as $m \ra \infty,$ $|O_m|\ra 0$ and hence $| \Gamma (E)-\tilde{\Gamma} (E)| \leq \epsilon L $, where $\epsilon>0$ is arbitrary, i.e., $\Gamma(E)=\tilde{\Gamma} (E).$}

\noi {\bf{Step 3:}} $ \tilde{\Gamma} \geq |D^s u|^p + \frac{\nu}{C_{H} \C_w^*},$ if $\C_w^* \neq 0$. Let $\C_w^* \neq 0$. Then from Lemma \ref{mlc1} we have  $ \Gamma \geq \frac{\nu}{C_{H} \C_w^*}$. Furthermore, \eqref{21} and \eqref{measureinequality} ensures that $\nu$ is supported on $\sum_w.$ Hence Step 1 and Step 2 yields the following: 
\begin{equation}\label{rep1}
 \tilde{\Gamma} \geq \left\{\begin{array}{ll}
    |D^s u|^p,  &  \\
     \frac{\nu}{C_{H} \C_w^*} .
\end{array}\right.    
\end{equation}
Since $|\overline{\sum_w}|=0$, the measure $|D^s u|^p$ is supported inside $\overline{\sum_w}^{c}$ and hence from \eqref{rep1} we easily obtain  $\tilde{\Gamma} \geq |D^s u|^p +  \frac{\nu}{C_{H} \C_w^*}.$
\end{proof}
\noi Now we prove the following Lemma.
     \begin{lemma}\label{mlc2}
   Let  {$0<s<1<p<\frac{N}{s}$ and} $w \in \mathcal{H}_{s,p}(\R^N)$, $w \geq 0$ and $u_n \wra u$ in $\mathcal{D}^{s,p}(\R^N)$. Set
   \begin{eqnarray*}
    \nu_{\infty} = \displaystyle\lim_{R \ra \infty} \overline{\lim_{n \ra \infty}}  \nu_n( \overline{B_R}^c) \quad \mbox{and} \quad \Gamma_{\infty} =  \displaystyle\lim_{R \ra \infty} \overline{\lim_{n \ra \infty}}  \Gamma_n( \overline{B_R}^c).
   \end{eqnarray*} 
  Then
\begin{enumerate}[(i)]
 \item $\nu_{\infty} \leq C_{H} \C_w(\infty)  \Gamma_{\infty} \nonumber,$
  \item $ \displaystyle \overline{\lim}_{n \ra \infty} \int_{\R^N} w|u_n|^p\, \dx= \int_{\R^N} w|u|^p\, \dx+ \norm{\nu} + \nu_{\infty}$,
\item[(iii)] Further, if $|\overline{\sum_w}|=0$, then we have

\begin{equation*}
\displaystyle \overline{\lim}_{n \ra \infty} \int_{\R^N} |D^s u_n|^p \dx \geq \begin{cases}
 \int_{\R^N} |D^s u|^p \dx + \frac{\norm{\nu}}{C_H\C_w^*} +  \Gamma_{\infty}, \quad \ \text{if} \ \ \C_w^* \neq 0 \\
 \int_{\R^N} |D^s u|^p \dx  +  \Gamma_{\infty}, \quad \ \text{otherwise} \,,
     \end{cases}
  \end{equation*}
\end{enumerate}
where $C_H>0$ is as in Theorem \ref{H1}.
  \end{lemma}
\begin{proof}
(i) For $R>0$, choose $\Phi_R \in C_b^{1}(\R^N)$ satisfying  $0\leq \Phi_R \leq 1$, $\Phi_R = 0 $ on $\overline{B_R}$ and $\Phi_R = 1 $ on $B_{R+1}^c$. Clearly, $v_n:=(u_n-u) \Phi_R \in \mathcal{D}^{s,p}_0( \displaystyle\overline{B_R}^c)$. Since
$\norm{w \chi_{ \overline{B_R}^c}}_{\mathcal{H}_{s,p}} < \infty, $ by Maz'ya's theorem \tk{ and using Corollary \ref{passing_limit} we get}
\begin{align}
& \, \tk{\overline{\lim_{n \ra \infty}}} \int_{\R^N} |\Phi_R|^p \ \mathrm{d}\nu_n  = \tk{\overline{\lim_{n \ra \infty}}} \int_{\R^N} w|(u_n-u)\Phi_R|^p \ \dx  =\tk{\overline{\lim_{n \ra \infty}}} \int_{\overline{B_R}^c} w |v_n|^p \ \dx \nonumber \\  & \leq C_H   \norm{w \chi_{\overline{B_R}^c}}_{\mathcal{H}_{s,p}}  \tk{\overline{\lim_{n \ra \infty}}} \int_{\R^N \times \R^N} \frac{|v_n(x)-v_n(y)|^p}{|x-y|^{N+sp}} \ \dxy \nonumber \\
& \leq C_{H} \norm{w \chi_{\overline{B_R}^c}}_{\mathcal{H}_{s,p}} \tk{\overline{\lim_{n \ra \infty}}}\int_{\R^N \times \R^N} |\Phi_R(x)|^p \frac{|(u_n-u)(x)-(u_n-u)(y)|^p}{|x-y|^{N+sp}} \ \dxy \nonumber \\
&= C_{H} \norm{w \chi_{\overline{B_R}^c}}_{\mathcal{H}_{s,p}} \tk{\overline{\lim_{n \ra \infty}}}\int_{\R^N} |\Phi_R|^p \ \mathrm{d}\Gamma_n. \label{lab: 1} 
\end{align}
Also, notice that 
\begin{align*}
    \nu_n(\overline{B_{R+1}^c}) & \leq \int_{\R^N} |\Phi_R|^p \ \mathrm{d}\nu_n \leq \nu_n(\overline{B_{R}^c}), \, \\
    \Gamma_n(\overline{B_{R+1}^c}) & \leq \int_{\R^N} |\Phi_R|^p \ \mathrm{d}\Gamma_n \leq \Gamma_n(\overline{B_{R}^c}) \,.
\end{align*}
Thus,
\begin{equation} \label{nu infty gamma infty}
    \displaystyle\nu_{\infty}=\lim_{R \ra \infty} \overline{\lim_{n \ra \infty}} \int_{\R^N} |\Phi_R|^p \ \mathrm{d}\nu_n ~~\text{and}~~ \displaystyle\Gamma_{\infty}=\lim_{R \ra \infty} \overline{\lim_{n \ra \infty}}\int_{\R^N} |\Phi_R|^p \ \mathrm{d}\Gamma_n.
\end{equation}
\tk{Consequently, $R \ra \infty$} in \eqref{lab: 1}, we get $\nu_{\infty} \leq C_{H} \C_{w}(\infty) \Gamma_{\infty}.$
 
\noi $(ii)$ By choosing $\Phi_R$ as above and using Br\'ezis-Lieb lemma together with \eqref{nu infty gamma infty} we have
\begin{align*}
 \overline{\lim_{n\ra \infty}} \int_{\R^N} w| u_n|^p \dx
 &= \tk{\lim_{R\ra \infty}} \  \overline{\lim_{n\ra \infty}} \left[ \int_{\R^N} w| u_n|^p (1-\Phi_R)\dx + \int_{\R^N} w| u_n|^p \Phi_R \dx  \right] \\
&=\tk{\lim_{R\ra \infty}} \ \overline{\lim_{n\ra \infty}} \left[  \int_{\R^N} w| u|^p (1-\Phi_R)  \dx  + \int_{\R^N} w| u_n-u|^p (1-\Phi_R) \dx + \int_{\R^N} w| u_n|^p \Phi_R \dx \right]   \\
&= \int_{\R^N} w| u|^p \dx + \norm{\nu} + \nu_{\infty}. 
\end{align*} 
The last equality uses the facts that $\nu_n \overset{\ast}{\rightharpoonup} \nu$, and $\overline\lim_{n \ra \infty} (a_n + b_n) = \lim_{n \ra \infty}
a_n + \overline\lim_{n \ra \infty}
b_n$ for any real sequences $(a_n),(b_n)$ with $(a_n)$ being convergent.

\noi $(iii)$  Notice that 
 \begin{align}
\overline{\lim_{n\ra \infty}} \int_{\R^N} |D^s u_n|^p \dx &= \overline{\lim_{n\ra \infty}} \left[ \int_{\R^N} |D^s u_n|^p (1-\Phi_R) \dx + \int_{\R^N} |D^s u_n|^p \Phi_R \dx\right] \nonumber \\ &= \tk{\int_{\R^N} (1- \Phi_R) \ \mathrm{d}\tilde{\Gamma}}+ \overline{\lim_{n\ra \infty}}  \int_{\R^N}  \Phi_R \ \mathrm{d}\tilde{\Gamma}_n. \label{lab: 2}
  \end{align}
Now for a given $\epsilon>0$ we have, \begin{align*}
 &\,   \left|\int_{\R^N}   \Phi_R \, \mathrm{d}\Gamma_n \, \dx -\int_{\R^N}   \Phi_R \, \mathrm{d}\tilde{\Gamma}_n \, \dx \right| \\ &= \left|\int_{\R^N}   \Phi_R |D^s (u_n-u)|^p \, \dx -\int_{\R^N}   \Phi_R |D^s u_n|^p \, \dx \right| \\
   &\leq  \int_{\R^N \times \R^N} \Phi_R(x) \frac{\big| |u_n(x)-u_n(y)|^p - |(u_n-u)(x)-(u_n-u)(y)|^p \big|}{|x-y|^{N+sp}}\, \dxy \\
    &\leq  \epsilon \int_{\R^N \times \R^N} \Phi_R(x) \frac{ |u_n(x)-u_n(y)|^p}{|x-y|^{N+sp}}\, \dxy + c(\epsilon,p) \int_{\R^N \times \R^N} \Phi_R(x) \frac{ |u(x)-u(y)|^p}{|x-y|^{N+sp}}\, \dxy \\
    & \leq \epsilon \int_{\R^N \times \R^N} \frac{ |u_n(x)-u_n(y)|^p}{|x-y|^{N+sp}}\, \dxy + c(\epsilon,p) \int_{\overline{B_R}^c}  |D^s u|^p \, \dx \,.
\end{align*} 
\tk{As limit $n \ra \infty$, we get
\begin{align*}
    \left|\int_{\R^N}   \Phi_R \, \mathrm{d}\Gamma \ \dx -\int_{\R^N}   \Phi_R \, \mathrm{d}\tilde{\Gamma} \ \dx \right| \leq \epsilon L + c(\epsilon,p) \int_{\overline{B_R}^c}  |D^s u|^p \, \dx \,,
\end{align*} 
where $L=\sup_{n \in \N} \{\|u_n\|_{s,p}\} <\infty$.
}
Then, by taking $R \ra \infty$, we have $\left|\int_{\R^N}   \Phi_R \, \mathrm{d}\Gamma \ \dx -\int_{\R^N}   \Phi_R \, \mathrm{d}\tilde{\Gamma} \ \dx \right| \leq \epsilon L$, where $\epsilon>0$ is arbitrary. Hence, $\Gamma_{\infty}=\displaystyle \lim_{R \ra \infty} \int_{\R^N}   \Phi_R \ \mathrm{d}\tilde{\Gamma} \ \dx=\displaystyle \lim_{R \ra \infty} \overline{\lim_{n \ra \infty}} \int_{\R^N}   \Phi_R \ \mathrm{d}\tilde{\Gamma}_n \ \dx$.
Therefore, by taking $R \ra \infty$ in \eqref{lab: 2}, we get
\begin{equation*} 
 \overline{\lim_{n\ra \infty}} \int_{\R^N} |D^s u_n|^p \ \dx = \|\tilde{\Gamma}\|+ \Gamma_{\infty}.   
\end{equation*}
Now, using Lemma \ref{lemma9}, we obtain \begin{equation*}
\displaystyle \overline{\lim}_{n \ra \infty} \int_{\R^N} |D^s u_n|^p \ \dx \geq \begin{cases}
\displaystyle \int_{\R^N} |D^s u|^p \dx + \frac{\norm{\nu}}{C_H\C_g^*} +  \Gamma_{\infty}, \quad \ \text{if} \ \C_g^* \neq 0 \\
\displaystyle \int_{\R^N} |D^s u|^p \dx +  \Gamma_{\infty}, \quad \ \text{otherwise}. 
     \end{cases}
  \end{equation*} 
\end{proof} 

\subsection{Proof of Theorem \ref{allinone1}}
  {In this subsection, we  prove the equivalence given by Theorem \ref{allinone1}. We start with the following key lemma that gives a necessary condition for the compactness of $W$ on $\mathcal{D}^{s,p}(\R^N)$.}
 \begin{lem} \label{cpct}
 Let  {$0<s<1<p<\frac{N}{s}$ and} $w \in \mathcal{H}_{s,p}(\R^N)$ and $W(u):=\displaystyle\int_{\R^N} |w| |u|^p \dx$ on $\mathcal{D}^{s,p}(\R^N)$ is compact. Then, 
 \begin{enumerate}
  \item[(i)] if $(A_n)$ is a sequence of bounded measurable subsets such that $\chi_{A_n}$ decreases to $0,$ then $$\displaystyle \norm{w \chi_{A_n}}_{\mathcal{H}_{s,p}} \ra 0 \text{ as } n \ra \infty.$$
  \item[(ii)] $\norm{w \chi_{B_n^c}}_{\mathcal{H}_{s,p}} \ra 0$
  as $n \ra \infty$.
 \end{enumerate}
 \end{lem}
\begin{proof} $(i)$ Let $(A_n)$ be a sequence of bounded measurable subsets such that $\chi_{A_n}$ decreases to $0$. Suppose that,  $\norm{w \chi_{A_n}}_{\mathcal{H}_{s,p}} \nrightarrow 0$. Then, there exists $a>0$ such that $\norm{w \chi_{A_n}}_{\mathcal{H}_{s,p}} >a$, for all $n$ (by the monotonicity of the norm).
  Thus, there exists a sequence of compact sets $F_n \subset \R^N$ and $ u_n \in  \mathcal{N}_{s,p}(F_n)$ with $0 \leq u_n \leq 1$ such that  
 \begin{equation} \label{convDp}
 \int_{\R^N} |D^s u_n|^p \dx < \frac{1}{a} \int_{F_n \cap A_n} |w| \dx \leq \frac{1}{a}\int_{|u_n| = 1} |w| |u_n|^{p_{s}^*} \dx \,.
 \end{equation}
 Since $A_n$'s are bounded and $\chi_{A_n}$ decreases to $0$, it follows that $|A_n| \ra 0$, as $n \ra \infty.$ Hence, we also have $\int_{F_n \cap A_n} |w| \ \dx \ra 0$ as $n \ra \infty$ (as $w \in L^1(A_1)$). 
 Hence from the above inequalities, $u_n \ra 0$ in $\mathcal{D}^{s,p}(\R^N)$. Now take $v_n=\frac{u_n^{\frac{p_{s}^*}{p}}}{\|u_n\|_{s,p}}$. Then, one can show that $(v_n)$ is bounded in $\mathcal{D}^{s,p}(\R^N)$ and  $v_n \ra 0$ a.e. because $\|v_n\|_p^p=\frac{\|u_n\|_{p_{s}^*}^{p_{s}^*}}{\|u_n\|_{s,p}^{p}} \leq C\|u_n\|_{s,p}^{p_{s}^*-p} \ra 0$ as $n\ra \infty$. Thus, $v_n \wra 0$ in $\mathcal{D}^{s,p}(\R^N)$.   By the compactness of $W$ we infer that $\lim_{n\ra \infty} \int_{\R^N} |w||v_n|^p\dx =0.$
 On the other hand,  
 \begin{eqnarray*}
  \int_{\R^N} |w| |v_n|^p\, \dx &=& \int_{\R^N}  \frac{|w| |u_n|^{p_{s}^*}}{ \norm{u_n}_{s,p}^p}\, \dx
   \geq  \int_{|u_n| = 1} \frac{|w| |u_n|^{p_{s}^*}}{ \norm{u_n}_{s,p}^p}\, \dx >a \,,
 \end{eqnarray*}
which is a contradiction.

\noi $(ii)$ If  $\norm{w \chi_{B_n^c}}_{\mathcal{H}_{s,p}} \nrightarrow 0$, as $n \ra \infty,$ then there exists a sequence of compact sets $F_n \subset \R^N$ such that
   \begin{eqnarray*}
a < \frac{\int_{F_n \cap B_n^c}  |w|\, \dx}{\cp(F_n)} \leq \frac{\int_{F_n \cap B_n^c} |w|\, \dx}{\cp(F_n \cap B_n^c)} 
\leq \frac{C\int_{F_n \cap B_n^c} |w|\, \dx}{\cp(F_n \cap B_n^c, \overline{B}_{\frac{n}{2}}^c) }
\end{eqnarray*}
 for some $a>0$ and $C>0$. The last inequality follows from part $(ii)$ of Proposition \ref{propofcap}.
 Thus, for each $n\in \N$ there exists $z_{n} \in \mathcal{N}_{s,p}( F_n \cap B_n^c, \overline{B}_{\frac{n}{2}}^c)$   such that 
\[ \int_{\R^N} |D^s z_n|^p\, \dx < \frac{C}{a} \int_{F_n \cap B_n^c} |w| \,\dx \leq \frac{C}{a} \int_{\R^N} |w||z_n|^p\, \dx.\]
 By taking $v_n=\displaystyle \frac{z_n}{\norm{z_n}_{s,p}}$ and following a same argument as in $(i)$ we contradict the compactness of $W$.
\end{proof} 
Next for $\phi \in C_c^{}(\R^N)$ we compute $\C_{\phi}$. For that, we will be using the fact that
$$\cp((F \cap B_r)^{\star}) \geq \mathcal{K}_{N,s,p} \tk{d^{N-sp}} \,,$$
where $\mathcal{K}_{N,s,p}>0$ is a constant \tk{independent of $d$ \cite{Jie} and $d$ is the radius of $(F \cap B_r)^{\star}$}.
 \begin{proposition} \label{Cgzero}
  Let $\phi \in C_c(\R^N)$. Then $\C_{\phi} \equiv 0.$
 \end{proposition}
 \begin{proof}
\tk{Let $F$ be a compact set in $\R^N$ such that $|F|\neq 0$}. Then for $\phi \in C_c^{}(\R^N)$, we have
\begin{eqnarray*} 
  \norm{\phi \chi_{B_r(x)}}_{\mathcal{H}_{s,p}} =  \tk{\sup_{F}} \left[ \frac{\int_{F \cap B_r(x)}|\phi|\dx}{\cp(F, \R^N)}\right]  
 \leq  \tk{\sup_{F}} \left[ \frac{ \sup (|\phi|) |(F\cap B_r)^{\star}|}{\cp((F \cap B_r)^{\star})}  \right]. \nonumber
 \end{eqnarray*}
 The last inequality follows from P\'olya-Szeg\"o inequality \cite[Theorem 9.2]{Almgren}.
 If $d$ is the radius of $(F \cap B_r)^{\star}$ then
 \[\frac{ |(F\cap B_r)^{\star}|}{\cp((F \cap B_r)^{\star})} \leq \text{C}(N,s,p) \frac{d^N}{d^{(N-sp)}} = \text{C}(N,s,p) d^{sp} \leq \text{C}(N,s,p) r^{sp} .\]
 Thus, $\C_{\phi}(x) = \lim_{r \ra 0} \norm{\phi \chi_{B_r(x)}}_{\mathcal{H}_{s,p}} =0 $. Also, one can easily see that $\C_{\phi}(\infty)=0 $ as $\phi$  has compact support.
 \end{proof}

Now, we prove our main theorem. 
\begin{proof}[Proof of Theorem \ref{allinone1}]
 $(i) \implies (ii):$ Let $W$ be compact. Take a sequence of measurable subsets $(A_n)$ 
 of $\R^N$ such that $\chi_{A_n}$ decreases to $0$ a.e. in $\R^N$. Part $(ii)$ of Lemma \ref{cpct} gives $\norm{w \chi_{B_n^c}}_{\mathcal{H}_{s,p}} \ra 0$, as $n \ra \infty$. Choose $\epsilon >0$ arbitrarily. There exists $N_0 \in \N,$ 
such that $\norm{w \chi_{B_n^c}}_{\mathcal{H}_{s,p}} \leq \frac{\epsilon}{2},$ for all $n \geq N_0.$ Now $A_n= (A_n \cap B_{N_0}) \cup (A_n \cap B_{N_0}^c)$,  for each $n$. Thus, 
\[\norm{w \chi_{A_n}}_{\mathcal{H}_{s,p}} \leq \norm{w \chi_{A_n \cap B_{N_0}}}_{\mathcal{H}_{s,p}} + \norm{w \chi_{A_n \cap B_{N_0}^c}}_{\mathcal{H}_{s,p}} \leq \norm{w \chi_{A_n \cap B_{N_0}}}_{\mathcal{H}_{s,p}} + \frac{\epsilon}{2}.\]
 Part $(i)$ of Lemma \ref{cpct} implies that there exists a natural number $N_1(\geq N_0)$ such that $$\norm{w \chi_{A_n \cap B_{N_0}}}_{\mathcal{H}_{s,p}} \leq \frac{\epsilon}{2}, \ \forall n \geq N_1 \,,$$ and hence $\norm{w \chi_{A_n}}_{\mathcal{H}_{s,p}} \leq \epsilon$ for all $n \geq N_1$. 
 Therefore, $w$ has an absolutely continuous norm.

\noi $(ii) \implies (iii):$ Let $w$ has absolute continuous norm in $\mathcal{H}_{s,p}$. Then, $\norm{w \chi_{B_m^c}}_{\mathcal{H}_{s,p}} $ converges to $0$ as $m \ra \infty$. Let $\epsilon >0$ be arbitrary. We choose $m_{\var} \in \N$ such that $\norm{w \chi_{B_m^c}}_{\mathcal{H}_{s,p}} < \epsilon$, for all $m \geq m_{\var}$. Now for any $n \in \N$,
\[w = w \chi_{\{|w| \leq n\} \cap B_{m_{\var}}} + w \chi_{\{|w| >n\} \cap B_{m_{\var}}} + w \chi_{B_{m_{\var}}^c} := w_n + z_n.\]
where $w_n=w \chi_{\{|w| \leq n\} \cap B_{m_{\var}}}$ and $z_n=w \chi_{\{|w| >n\} \cap B_{m_{\var}}} + w \chi_{B_{m_{\var}}^c}.$ Clearly, $w_n \in L^{\infty}(\R^N)$ and $|supp(w_n)| < \infty $. 
Furthermore,
\[ \norm{z_n}_{\mathcal{H}_{s,p}} \leq \norm{w \chi_{\{|w| >n\} \cap B_{m_{\var}}}}_{\mathcal{H}_{s,p}} + \norm{w \chi_{B_{m_{\var}}^c}}_{\mathcal{H}_{s,p}} <  \norm{w \chi_{\{|w| >n\} \cap B_{m_{\var}}}}_{\mathcal{H}_{s,p}} + \epsilon \,. \]
 Now, $w \in L^1_{\mathrm{loc}}(\R^N)$ ensures that $\chi_{\{|w| >n\} \cap B_{m_{\var}}} \ra 0$ as $n\ra \infty$. As $w$ has absolutely continuous norm, $\norm{w \chi_{\{|w| >n\} \cap B_{{m_{\var}}}}}_{\mathcal{H}_{s,p}} < \epsilon$ for large $n$. Therefore, $\norm{z_n}_{\mathcal{H}_{s,p}}< 2\epsilon$ for large $n$.
 Hence, Lemma \ref{charF} concludes that $w \in \mathcal{H}_{s,p,0}(\R^N)$.
 
 \noi $(iii) \implies (iv):$ Let $w \in \mathcal{H}_{s,p,0}(\R^N)$ and $\epsilon >0$ be arbitrary. Then there exists $w_\var \in C_c(\R^N)$ such that
 $\norm{w-w_\var}_{\mathcal{H}_{s,p}} < \epsilon$. Thus Proposition \ref{Cgzero} infers that $\C_{w_\var}$ vanishes.
Now as $w = w_\var + (w-w_\var)$, it follows that $\C_w(x) \leq \C_{w_\var}(x) + \C_{w - w_\var}(x) \leq \norm{w - w_\var}_{\mathcal{H}_{s,p}} < \epsilon$ and hence $\C^*_w=0$. By a similar argument one can show $\C_w(\infty)=0.$

\noi $(iv) \implies (i):$ Assume that $\C^*_w =0= \C_w(\infty)$. Let $(u_n)$ be a bounded sequence in $ \mathcal{D}^{s,p}(\R^N)$. Then by Lemma \ref{mlc2}, up to a sub-sequence we have,
 \begin{eqnarray*}
 \nu_{\infty} &\leq& C_H\ \C_w(\infty)  \Gamma_{\infty} \label{1},\\
 \norm{\nu} &\leq& C_H \C^{*}_w \norm{\Gamma} \label{2}, \\
 \tk{\lim_{n \ra \infty}} \int_{\R^N} |w||u_n|^p \dx &=& \int_{\R^N} |w||u|^p \dx + \norm{\nu} + \nu_{\infty} \label{3}.
 \end{eqnarray*}
 As $\C^*_w=0= \C_w(\infty)$ we immediately conclude that
 $\displaystyle \tk{\lim_{n \ra \infty}} \int_{\R^N} |w||u_n|^p\, \dx = \int_{\R^N} |w||u|^p\,\dx $
 and hence $W: \mathcal{D}^{s,p}(\R^N) \mapsto \R$ is compact.
 \end{proof} 
\section{Weighted Eigenvalue Problem} \label{WEVP}
This section deals with the weighted eigenvalue problem given by \eqref{Weighetd eigenvalue problem}. We show the existence of the first eigenvalue by using   {variational techniques on the} Rayleigh quotient and then prove some qualitative properties of the first eigenvalue. Finally, we prove  {the existence of infinitely many eigenvalues that increase to infinity.} 

\subsection{Qualitative behaviour of the first eigenvalue}
We show the existence of an eigenvalue by following a direct variational approach. We begin with the Rayleigh quotient $Q(u)$ given by
\begin{equation} \label{Rayleigh quotient}
   Q(u) = \dfrac{\int_{\R^N \times \R^N} \frac{|u(x)-u(y)|^{p}}{|x-y|^{N+sp}}\, \dxy}{\int_{\mathbb{R}^{N}} w|u|^{p}\, \dx},
\end{equation}
with the domain of definition 
\begin{equation} \label{L}
    \mathbb{L}:= \{ u \in \D^{s,p}(\mathbb{R}^{N}): \int_{\mathbb{R}^{N}} w|u|^{p}\, \dx > 0 \} \,.
\end{equation}
Since $w \in L_{\mathrm{loc}}^{1}(\mathbb{R}^{N})$ and $w_{1} \not\equiv 0$, then by \cite[Proposition 4.2]{Prashanth} there exists $\phi \in C_{c}^{\infty}(\mathbb{R}^{N})$ such that $\int_{\mathbb{R}^{N}} w|\phi|^{p}dx >0$. Therefore, the set $\mathbb{L}$ is non-empty.
Now, let us consider 
\begin{equation} \label{M}
  \mathbb{S} := \{ u \in \D^{s,p}(\mathbb{R}^{N}) : \int_{\mathbb{R}^{N}} w|u|^{p}\, \dx = 1 \},
\end{equation}
\begin{equation} \label{J(u)}
 J(u) =  \int_{\mathbb{R}^{N} \times \mathbb{R}^{N}} \frac{|u(x)-u(y)|^{p}}{|x-y|^{N+sp}}\, \dxy.
\end{equation}
 If $Q$ is $C^{1}$, then the critical points of $Q$ over $\mathbb{L}$ correspond to the Euler-Lagrange equation associated with the weighted eigenvalue problem (\ref{Weighetd eigenvalue problem}), and the corresponding critical values of $Q$ are the eigenvalues of the problem (\ref{Weighetd eigenvalue problem}).  Observe that finding a critical point of the Rayleigh quotient $Q$ over the domain $\mathbb{L}$ is similar to finding the critical point of the functional $J$ over $\mathbb{S}$, i.e., there is a one-to-one correspondence between them. Therefore, we try to find the critical points of the functional $J$ on $\mathbb{S}$ by employing some sufficient  {conditions} on $w_1$. One of the main difficulties in showing the existence of a critical point of $J$ on $\mathbb{S}$ arises due to the non-compactness of the map $W$. Since we have a weak assumption on $w_{2}$, i.e., it is just locally integrable, therefore the map $W: \D^{s,p}(\mathbb{R}^{N}) \rightarrow \mathbb{R}$ given by
\begin{equation} \label{W(u)}
W(u) = \int_{\mathbb{R}^{N}} w|u|^{p}\,\dx,
\end{equation}
may not even be continuous and hence $\mathbb{S}$ may not be closed in $\D^{s,p}(\mathbb{R}^{N})$. Despite that, we prove that a sequence of minimizers of $J$ on $\mathbb{S}$ has a weak limit, which also lies in $\mathbb{S}$. From the definition of the space $\D^{s,p}(\mathbb{R}^{N})$, it is easy to check that the functional $J$ becomes coercive and weakly lower semi-continuous on $\D^{s,p}(\mathbb{R}^{N})$. Further, if $w_{1} \in \mathcal{H}_{s,p,0}(\mathbb{R}^{N})$, the map $W_{1}: \D^{s,p}(\mathbb{R}^{N}) \rightarrow \mathbb{R}$ given by
\begin{equation} \label{W1(u)}
W_{1}(\varphi) = \int_{\mathbb{R}^{N}} w_{1}|\varphi|^{p}\,\dx,
\end{equation}
is compact on $\D^{s,p}(\mathbb{R}^{N})$ by Theorem \ref{allinone1}.

\begin{theorem}
Let $w \in L_{\mathrm{loc}}^{1}(\mathbb{R}^{N})$ with $w_{1} \in \mathcal{H}_{s,p,0}(\mathbb{R}^{N}), w_1 \not\equiv 0$ and \tk{$0<s<1<p<N/s$}. Then $J$ admits a minimizer on $\mathbb{S}$.
\end{theorem}

\begin{proof}
Since $w \in L_{\mathrm{loc}}^{1}(\mathbb{R}^{N})$ and $w_{1} \not\equiv 0$, then by \cite[Proposition 4.2]{Prashanth} there exists $\phi \in C_{c}^{\infty}(\mathbb{R}^{N})$ such that $\int_{\mathbb{R}^{N}} w|\phi|^{p}dx >0$
 and hence $\mathbb{S} \neq \emptyset$. Let $\{u_{n}\}$ be a minimizing
sequence for $J$ on $\mathbb{S}$; i.e.,
    \begin{equation*}
    \lim_{n \rightarrow \infty} J(u_{n}) = \lambda_{1} := \inf_{u \in \mathbb{S}} J(u).
    \end{equation*}
\tk{Since} $\{u_{n}\}$ is bounded in $\D^{s,p}(\mathbb{R}^{N})$, by reflexivity $\{u_{n}\}$ admits a weakly convergent subsequence in $\D^{s,p}(\mathbb{R}^{N})$. Let us denote the subsequence by $\{u_{n}\}$ itself and the weak limit by $u$ in $\D^{s,p}(\mathbb{R}^{N})$. Further, the compactness of the map $W_{1}$ gives
     \begin{equation*}
    \lim_{n \rightarrow \infty} \int_{\mathbb{R}^{N}} w_{1}|u_{n}|^{p}\,\dx = \int_{\mathbb{R}^{N}} w_{1}|u|^{p}\,\dx.
     \end{equation*}
Since $u_{n} \in \mathbb{S}$, we write 
     \begin{equation*}
    \int_{\mathbb{R}^{N}} w_{2}|u_{n}|^{p}\,\dx = \int_{\mathbb{R}^{N}} w_{1}|u_{n}|^{p}\,\dx -1.
     \end{equation*}
Also, we know that the embedding $\D^{s,p}(\mathbb{R}^{N}) \hookrightarrow L_{\mathrm{loc}}^{p}(\mathbb{R}^{N})$ is compact, thus $u_{n} \rightarrow u$ a.e. in $\mathbb{R}^{N}$ up to a subsequence. We apply Fatou's Lemma to get
     \begin{equation*}
    \int_{\mathbb{R}^{N}} w_{2}|u|^{p}\,\dx \leq \int_{\mathbb{R}^{N}} w_{1}|u|^{p}\,\dx -1,
     \end{equation*}
which shows that $ \int_{\mathbb{R}^{N}} w|u|^{p}\dx \geq 1$. Setting $\tilde{u} := \dfrac{u}{(\int_{\mathbb{R}^{N}} w|u|^{p}\,\dx )^{1/p}}$, and since $J$ is weakly lower semi-continuous, we have 
    \begin{equation*}
    \lambda_{1} \leq J(\Tilde{u}) = \dfrac{J(u)}{\int_{\mathbb{R}^{N}} w|u|^{p}\, \dx} \leq J(u) \leq \liminf_{n} J(u_{n}) = \lambda_{1}.
    \end{equation*}
Thus the equality must hold at each step and $\int_{\mathbb{R}^{N}} w|u|^{p}\dx =1 $, which shows that $u \in \mathbb{S}$ and $ J(u) = \lambda_{1}$. Hence, $J$ admits a minimizer $u$ on $\mathbb{S}$.
\end{proof}
Further, we prove that any minimizer of $Q$ on $\mathbb{L}$ is an eigenfunction of \eqref{Weighetd eigenvalue problem}.
\begin{proposition} \label{u is a eigenfunction}
Let $u$ be a minimizer of $Q$ on $\mathbb{L}$. Then $u$ is an eigenfunction of (\ref{Weighetd eigenvalue problem}). \tk{In particular, $u$ is an first eigenfunction of (\ref{Weighetd eigenvalue problem}), i.e., an eigenfunction corresponding to the first eigenvalue $\lambda_{1} := \inf_{u \in \mathbb{S}} J(u)$.}
\end{proposition}
\begin{proof}
For each $\phi \in C_{c}^{\infty}(\mathbb{R}^{N})$, we can verify that $Q$ admits directional derivative along $\phi$ by using the dominated convergence theorem. It is given that $u$ is a minimizer of $Q$ on $\mathbb{L}$, therefore we have a necessary condition
      \begin{align*}
    \dfrac{\mathrm{d}}{\mathrm{d}t} Q(u + t \phi) |_{t = 0} &= 0.\end{align*} 
   This further implies 
   \begin{align*} \int_{\mathbb{R}^{N} \times \mathbb{R}^{N}} \dfrac{|u(x)-u(y)|^{p-2} (u(x)-u(y)) (\phi(x) -\phi(y))}{ |x-y|^{N+sp}}\, \dxy &= \lambda_{1} \int_{\mathbb{R}^{N}} w|u|^{p-2}u\phi\,\dx, 
     \end{align*}
     $\text{for all}~\phi \in C_{c}^{\infty}(\mathbb{R}^{N})$, where $\lambda_{1} := \inf_{u \in \mathbb{S}} J(u)$. Now using the density of $ C_{c}^{\infty}(\mathbb{R}^{N})$ into $\D^{s,p}(\mathbb{R}^{N})$, we can conclude
    \begin{equation*}
    \int_{\mathbb{R}^{N} \times \mathbb{R}^{N}} \dfrac{|u(x)-u(y)|^{p-2} (u(x)-u(y)) (\phi(x) -\phi(y))}{ |x-y|^{N+sp}}\, \dxy = \lambda_{1} \int_{\mathbb{R}^{N}} w|u|^{p-2}u\phi\,\dx, 
    \end{equation*}
    $\text{for all}~ \phi \in \D^{s,p}(\mathbb{R}^{N}).$
\end{proof}
Next, we prove that the first eigenfunction does not change its sign. We adapt the idea from the article \cite{Cui} to prove this Lemma. 
\begin{lemma} \label{Positive}
The first eigenfunctions (i.e., the eigenfunctions corresponding to the first eigenvalue $\lambda_1$) of the weighted eigenvalue problem (\ref{Weighetd eigenvalue problem}) are of a constant sign. Moreover, a non-negative first eigenfunction is positive.
\end{lemma}
\begin{proof}
We consider $u_{1}$ the first eigenfunction of \eqref{Weighetd eigenvalue problem} corresponding to the first eigenvalue $\lambda_1$. Then $u_1$ is a minimizer of $J$ over $\mathbb{S}$. Since $u_1 \in \mathbb{S}$, this implies that $|u_1|\in \mathbb{S}$. Now we have
\begin{align*}
    \lambda_1 = \inf\limits_{u \in \mathbb{S}} \int_{\R^N \times \R^N} \frac{|u(x)-u(y)|^{p}}{|x-y|^{N+sp}}\,\dxy &\leq \int_{\R^N \times \R^N} \frac{\big||u_1(x)|-|u_1(y)|\big|^{p}}{|x-y|^{N+sp}}\,\dxy \\
    &\leq \int_{\R^N \times \R^N} \frac{|u_1(x)-u_1(y)|^{p}}{|x-y|^{N+sp}}\,\dxy = \lambda_1,
\end{align*}
Therefore, equality must hold at each step, which implies either $u_{1}^{+} \equiv 0$ or $u_{1}^{-} \equiv 0$. Thus, the eigenfunction $u_{1}$ corresponding to the first eigenvalue $\lambda_1$ does not change its sign.
 If we assume $u_1 \geq 0$, then we have
\begin{equation*}
     (-\Delta_{p})^{s}u_{1} + \lambda_{1}w_{2} (u_{1})^{p-1} = \lambda_{1}w_{1} (u_{1})^{p-1} \geq 0 ~~\text{in}~\mathbb{R}^{N}.
 \end{equation*}
Thus the strong minimum principle \cite[Theorem 1.2]{Pezzo} yields $u_{1} >0$ a.e. in $\mathbb{R}^{N}$.
\end{proof}
{Further, we show the simplicity of the first eigenvalue of \eqref{Weighetd eigenvalue problem}.}
\begin{lemma} \label{Simple}
The eigenfunction of (\ref{Weighetd eigenvalue problem}) corresponding to $\lambda_{1}$ are unique up to some constant multiplication, i.e., $\lambda_{1}$ is simple.
\end{lemma}
\begin{proof}
Let $\phi_{1}$ and $\phi_{2}$ be two eigenfunctions corresponding to the same eigenvalue $\lambda_{1}$, then we may suppose that $\phi_{1}$,~$\phi_{2}>0$ and $\phi_{1}, \phi_{2} \in \mathbb{S}$, namely,
\begin{equation*}
    \int_{\mathbb{R}^{N}} w|\phi_{1}|^{p} \dx = \int_{\mathbb{R}^{N}} w|\phi_{2}|^{p} \dx = 1.
\end{equation*}
Let 
\begin{equation*}
    \Phi = \bigg(\dfrac{\phi_{1}^{p}+\phi_{2}^{p}}{2}\bigg)^{1/p},
\end{equation*}
then we have $\Phi \in \mathbb{S}$. Since the function $\alpha(r,s) := |r^{1/p} - s^{1/p}|^{p}$ is convex for $r,s>0$, we have the following inequality
\begin{equation*}
    \alpha\bigg(\frac{r_{1}+r_{2}}{2}, \frac{s_{1}+s_{2}}{2}\bigg) \leq \frac{1}{2}\alpha(r_{1},s_{1}) + \frac{1}{2}\alpha(r_{2},s_{2}) ,
\end{equation*}
where the equality holds only for $r_{1}s_{2} = r_{2}s_{1}$ {(see \cite[Lemma 13]{Lindgren2014})}. 
Therefore, according to the above inequality and $\phi_{1},\phi_{2},\Phi \in \mathbb{S}$ we deduce,
\begin{align*}
    \lambda_{1} \leq \int_{\mathbb{R}^{N} \times \mathbb{R}^{N}} \dfrac{|\Phi(x) - \Phi(y)|^{p}}{|x-y|^{N+sp}}\, \dxy \leq \frac{1}{2}\int_{\mathbb{R}^{N} \times \mathbb{R}^{N}} \dfrac{|\phi_{1}(x) - \phi_{1}(y)|^{p}}{|x-y|^{N+sp}}\,\dxy
    + \frac{1}{2} \int_{\mathbb{R}^{N} \times \mathbb{R}^{N}} \dfrac{|\phi_{2}(x) - \phi_{2}(y)|^{p}}{|x-y|^{N+sp}}\dxy = \lambda_{1} \,. 
\end{align*}
Thus, equality must hold at each step. Therefore, $\phi_{1}(x)\phi_{2}(y) = \phi_{1}(y)\phi_{2}(x)$ which implies that $\phi_{1}(x) = c \phi_{2}(x)$ with $c \in \mathbb{R}$.
\end{proof}

The following Lemma shows that the first eigenfunctions are the only eigenfunctions that do not change their sign.
We use the idea of \cite[Theorem 3.3]{Goyal2018eigenvalues} to prove the following Lemma. To prove this result, we assume that $w_2$ is a Hardy-weight; not only just a locally integrable function.
\begin{lem}
\tk{Let $w=w_1-w_2$, $w_1,w_2 \geq 0$ be such that $w_1 \in \mathcal{H}_{s,p,0}(\mathbb{R}^{N}) \text{ and }w_2 \in \mathcal{H}_{s,p}(\mathbb{R}^{N})$}. Then, the eigenfunctions of \eqref{Weighetd eigenvalue problem} corresponding to eigenvalues other than $\lambda_1$ must change sign.
\end{lem} 
\begin{proof}
We assume that $u_1$ and $u$ are eigenfunctions corresponding to distinct eigenvalues 
$\lambda_1$ and $\lambda$, respectively. Then the following holds.
\begin{align}
     (-\Delta_{p})^{s}u_{1} &= \lambda_{1}w (u_{1})^{p-1} ~~\text{in}~(\D^{s,p}(\mathbb{R}^{N}))', \label{equation first}\\
      (-\Delta_{p})^{s}u_{} &= \lambda_{}w |u_{}|^{p-2}u ~~\text{in}~(\D^{s,p}(\mathbb{R}^{N}))'. \label{equation second}
\end{align}
We proceed by contradiction. Suppose, on the contrary, that the eigenfunction $u$ does not change sign. Without loss of generality, we may assume that $u \ge 0$. \tk{Let $\{\phi_m\} \subset C^\infty_c(\mathbb{R}^N)$ be a sequence} such that  $\phi_m \to u_1$ in $\D^{s,p}(\mathbb{R}^N)$ as $m \to \infty$. For some $\epsilon > 0$, we choose the test functions
\begin{align*}
    \xi_1 = u_1,
\qquad
\tk{\xi_2 = \frac{|\phi_m|^p}{(u + \epsilon)^{p-1}}}.
\end{align*}
First, we show that $\xi_2 \in \D^{s,p}(\mathbb{R}^N)$. To this end, we observe that \tk{
\begin{align*}
 &|\xi_2(x)-\xi_2(y)| \\
    &= \left|\tfrac{|\phi_m(x)|^p}{(u(x)+\epsilon)^{p-1}} - \tfrac{|\phi_m(y)|^p}{(u(y)+\epsilon)^{p-1}} \right|= \left|\tfrac{|\phi_m(x)|^p- |\phi_m(y)|^p}{(u(x)+\epsilon)^{p-1}} + \tfrac{|\phi_m(y)|^p \left((u(y)+\epsilon)^{p-1} - (u(x)+\epsilon)^{p-1} \right)}{(u(x)+\epsilon)^{p-1}(u(y)+\epsilon)^{p-1}} \right|\\
    & \leq \epsilon^{1-p}\big| \left|\phi_m(x) \right|^p- \left|\phi_m(y)\right|^p \big| + \norm{\phi_m}_{\infty}^p \tfrac{\big|(u(x)+\epsilon)^{p-1} - (u(y)+\epsilon)^{p-1} \big|}{(u(x)+\epsilon)^{p-1}(u(y)+\epsilon)^{p-1}}\\
     & \leq p\epsilon^{1-p}(|\phi_m(x)|^{p-1}+ |\phi_m(y)|^{p-1})|\phi_m(x)- \phi_m(y)| + \norm{\phi_m}_{\infty}^p (p-1)\tfrac{\big((u(x)+\epsilon)^{p-2} + (u(y)+\epsilon)^{p-2}\big)}{(u(x)+\epsilon)^{p-1}(u(y)+\epsilon)^{p-1}} \left|u(x) -u(y)\right|\\
    & \leq 2p\epsilon^{1-p}\norm{\phi_m}_{\infty}^{p-1}|\phi_m(x)- \phi_m(y)| + \norm{\phi_m}_{\infty}^p (p-1)\left( \tfrac{1}{(u(x)+\epsilon)(u(y)+\epsilon)^{p-1}} +\tfrac{1}{(u(x)+\epsilon)^{p-1}(u(y)+\epsilon)}  \right)|u(x) -u(y)|   .\end{align*} }
Consequently, it follows from the above estimate that
\begin{align*}
    |\xi_2(x)-\xi_2(y)|  \leq C(\epsilon,p,\norm{\phi_m}_{\infty}) \left(|\phi_m(x)- \phi_m(y)| + |u(x) -u(y)| \right).
\end{align*}
\tk{Since $\|\phi_m\|_{s,p}<\infty$ and $\|u\|_{s,p}<\infty$, it follows that 
$\|\xi_2\|_{s,p}<\infty$.  
Moreover, $\xi_2 \in L^{p_s^*}(\mathbb{R}^N)$.  
By the characterization of $\D^{s,p}(\mathbb{R}^N)$ (see \cite[Theorem 3.1]{Brasco-Castro-Vazquez-2021-CVPDE}), we get 
$\xi_2 \in \D^{s,p}(\mathbb{R}^N)$.
} Choosing $\xi_1$ and $\xi_2$ as test functions in \eqref{equation first} 
and \eqref{equation second}, respectively, we obtain
\begin{equation} \label{equation third}
    \int_{\mathbb{R}^{N} \times \mathbb{R}^{N}} \dfrac{|u_1(x)-u_1(y)|^{p}}{ |x-y|^{N+sp}}\, \dxy = \lambda_{1} \int_{\mathbb{R}^{N}} w|u_1|^{p}\, \dx,
\end{equation} and
\begin{align} \label{equation sp}
\begin{split}
    \int_{\mathbb{R}^{N} \times \mathbb{R}^{N}} \tfrac{|u(x)-u(y)|^{p-2} (u(x)-u(y))}{ |x-y|^{N+sp}} \tk{\left(\tfrac{|\phi_m(x)|^p}{(u(x)+\epsilon)^{p-1}}-\tfrac{|\phi_m(y)|^p}{(u(y)+\epsilon)^{p-1}} \right)}\, \dxy 
    = \lambda_{} \int_{\mathbb{R}^{N}} w|u|^{p-2}u \tk{\tfrac{|\phi_m(x)|^p}{(u(x)+\epsilon)^{p-1}}}\,\dx.
\end{split}
\end{align}
By Lemma~\ref{Picone identity}, we have $K(|\phi_m|, u + \epsilon) \ge 0$, 
where $K$ is given by \eqref{def:K}. Combining this inequality with \eqref{equation sp} yields
\begin{equation}
    \int_{\mathbb{R}^{N} \times \mathbb{R}^{N}} \dfrac{|\phi_m(x)-\phi_m(y)|^{p}}{ |x-y|^{N+sp}} \dxy  - \lambda_{} \int_{\mathbb{R}^{N}} w |\phi_m|^{p} \tk{\left(\frac{u}{u+ \epsilon} \right)^{p-1}}\, \dx \geq 0.
\end{equation}
Taking the limit as $\varepsilon \to 0$ and applying the dominated convergence theorem, we obtain \tk{
\begin{equation} \label{equation four}
    \int_{\mathbb{R}^{N} \times \mathbb{R}^{N}} \dfrac{|\phi_m(x)-\phi_m(y)|^{p}}{ |x-y|^{N+sp}} \dxy  - \lambda_{} \int_{\mathbb{R}^{N}} w |\phi_m|^{p}\, \dx \geq 0.
\end{equation}}
\tk{Since $w \in \mathcal{H}_{s,p}(\mathbb{R}^{N})$, $\phi_m \to u_1$ in $\D^{s,p}(\mathbb{R}^N)$ implies 
$$ \int_{\mathbb{R}^{N}} w |\phi_m|^{p}\, \dx \to \int_{\mathbb{R}^{N}} w u_1^{p}\, \dx$$
as $m \to \infty$.}
Thus, by subtracting \eqref{equation third} from \eqref{equation four} and 
letting $m \to \infty$, we deduce that
\begin{align*}
    (\lambda_1 -\lambda) \int_{\mathbb{R}^{N}}w |u_1|^p\, \dx \geq 0.
\end{align*}
Hence, the above inequality holds if and only if $\lambda_1 > \lambda$, which contradicts the fact that $\lambda_1$ is the smallest eigenvalue. This completes the proof.
\end{proof}
\subsection{Infinite set of eigenvalues}
{This section deals with the existence of an infinite set of eigenvalues of (\ref{Weighetd eigenvalue problem}). We
follow the Ljusternik-Schnirelmann theory on $C^1$-manifold due to Szulkin \cite{Szulkin}. The Ljusternik-Schnirelmann theory enables us to find the critical points of a functional $J$ on a manifold $M$. First, we recall the definition of the Palais-Smale (PS) condition and genus. Assume that $M$ is a $C^1$-manifold and $f \in C^{1}(M; \mathbb{R})$. A sequence $\{u_n\} \subset M$ is said to be a (PS) sequence on $M$ if $\{f(u_n)\}$ is bounded and $f'(u_n) \rightarrow 0$, where $f'(u)$ represents the Fr\'echet differential of $f$ at $u$. If every (PS) sequence $\{u_n\}$ admits a convergent subsequence, then we say $f$ satisfies the (PS) condition on $M$. Let $\Theta$ be the family of sets $A \subset M \setminus \{0\}$ such that $A$ is closed in $M$ and symmetric concerning $0$, i.e. $z \in A$ implies $-z \in A$. If $A \in \Theta$, then the Krasnoselski genus of $A$ is denoted by
$\gamma(A)$ and is defined as the smallest integer $k$ for which there exists a non-vanishing
odd continuous mapping from A to $\mathbb{R}^{k}$. When no such map exists for any $k$, we set $\gamma(A) = \infty$, and also we set $\gamma(\emptyset) = 0$.  We
refer to \cite{Rabinowitz} for more details and properties of the genus. We can deduce the next theorem from \cite[Corollary 4.1]{Szulkin}.}
\begin{theorem} \label{Ljusternik}
Let M be a closed symmetric $C^1$-submanifold of a real Banach space X and $0 \notin M$. Let $f \in C^1(M; \R)$ be an even function satisfying the (PS) condition
on M and is bounded below. Define
\[ \lambda_{j} := \inf_{A \in \Gamma_{j}} \sup_{x \in A} f(x),\]
where $\Gamma_j = \{ A \subset M : A ~\text{is compact and symmetric about origin},~ \gamma(A) \geq j\}$. If for a given $j,~ \lambda_j = \lambda_{j+1}= .~ .~ . = \lambda_{j+p} \equiv \lambda, ~then ~\gamma(K_\lambda) \geq p + 1,~ where~ K_\lambda = \{x \in M : f(x) = \lambda ,~ f'(x) = 0 \}$.
\end{theorem}
\noi It can be noticed that the set $ \mathbb{S} = \{ u \in \mathcal{D}^{s,p}(\mathbb{R}^{N}): \int_{\mathbb{R}^{N}} w|u|^{p} \dx = 1 \}$ may not admit a manifold structure from the topology on $\mathcal{D}^{s,p}(\mathbb{R}^{N})$, due to the weak assumptions on $w_{2}$. However, the set $\mathbb{S}$ inherits a $C^{1}$ Banach manifold structure from the following subspace $X$ of  $\mathcal{D}^{s,p}(\mathbb{R}^{N})$.

\noi For $w_{2} \in L^{1}_{\mathrm{loc}}(\mathbb{R}^{N})$, we define
\[ \|u\|_{X}^{p} : = \|u\|_{s,p}^p + \int_{\mathbb{R}^{N}} w_{2}|u|^{p} \dx, \] and
\[ X:= \{ u \in  \mathcal{D}^{s,p}(\mathbb{R}^{N}): \|u\|_{X} < \infty\}. \]
\begin{lemma}
The space $X = (X, \|\cdot\|_{X})$ is a uniformly convex Banach space.
\end{lemma}
\begin{proof} We divide the proof into several steps, following the approach used in \cite[Lemma 5.1]{Chen2022}.\\
\noi {\bf Step 1:} First, we show that $X$ is complete under the given norm. Let $\{u_{n}\}$ be a Cauchy sequence in $X$. That is, for every $\epsilon > 0$, there exists $N_0 \in \mathbb{N}$ depending on $\epsilon$ such that 
\begin{equation} \label{Cauchy}
    \|u_{n} - u_{m}\|_{X} < \epsilon,\quad \text{ for all }n,m \geq N_0.
\end{equation}
By the definition of the norm on $X$, we have $\|u_{n}-u_{m}\|_{s,p} \leq \|u_{n}-u_{m}\|_{X}< \epsilon.$ Hence, $\{u_{n}\}$ is also a Cauchy sequence in $\D^{s,p}(\mathbb{R}^{N})$. By completeness of $\D^{s,p}(\mathbb{R}^{N})$, there exists $u \in \D^{s,p}(\mathbb{R}^{N})$ such that $u_{n} \rightarrow u$ in $\D^{s,p}(\mathbb{R}^{N})$. It remains to verify that $u \in X$. From the convergence in $\D^{s,p}(\mathbb{R}^{N})$, there exists a subsequence $\{u_{n_{k}}\}$ such that $u_{n_{k}} \rightarrow u$ a.e. in $\mathbb{R}^{N}~\text{as}~k \rightarrow \infty$. Applying Fatou’s lemma together with \eqref{Cauchy}, we get
\begin{align*}
    \int_{\mathbb{R}^{N}} w_{2}|u|^{p} \dx &\leq \liminf_{k \rightarrow \infty} \int_{\mathbb{R}^{N}} w_{2}|u_{n_{k}}|^{p} \dx \\
   & \leq \liminf_{k \rightarrow \infty} ( \|u_{n_{k}}- u_{N_0}\|_{X} + \|u_{N_0}\|_{X})^{p}\\ &\leq (\epsilon + \|u_{N_0}\|_{X})^{p} < \infty.
\end{align*}
Thus, $u \in X$. Moreover, for all $n \geq N_0$, we have $ \|u_{n}-u\|_{X} \leq \liminf_{k \rightarrow \infty}\|u_{n}-u_{n_{k}}\|_{X} \leq \epsilon$. Consequently, $u_n \rightarrow u$ strongly in $X$, and therefore $X$ is complete. \\
{\bf Step 2:} Now we prove that $X$ is a uniformly convex Banach space. Fix $0 <\epsilon \leq 2$, let $u,v \in X$ satisfy 
\begin{align} \label{UC}
    \|u\|_{X}=1=\|v\|_{X} \ \text{and} \ \|u-v\|_{X} \geq \epsilon.
\end{align}
We treat the cases $1<p<2$ and $p \geq 2$ separately. We first consider $p \geq 2$. To this end, we recall the following inequality \cite[Lemma 2.37, page 42]{Adams2003sobolev}:
\begin{equation} \label{adam}
    \left| \frac{a+b}{2} \right|^{p} +  \left|\frac{a-b}{2} \right|^{p} \leq \frac{|a|^{p}+|b|^{p}}{2}, \ \text{for }a,b \in \mathbb{R}.
\end{equation}
From \eqref{adam}, we can deduce the following:
\begin{align*}
   \left\| \frac{u+v}{2} \right\|_{X}^{p} + \left\|\frac{u-v}{2} \right\|_{X}^{p} &=  \left\|{\frac{u+v}{2}}\right\|_{s,p}^{p} + \left\|{\frac{u-v}{2}}\right\|_{s,p}^{p} + \int_{\mathbb{R}^{N}} w_{2} \left(\left| \frac{u+v}{2}  \right|^{p} + \left| \frac{u-v}{2}  \right|^{p} \right) \dx\\
    &\leq \frac{1}{2} \left[  \|u\|_{s,p}^{p} + \|v\|_{s,p}^{p} + \int_{\mathbb{R}^{N}} w_{2} (|u|^{p}+ |v|^{p})\, \dx \right]  \\
   & = \frac{1}{2} [\|u\|_{X}^{p} + \|v\|_{X}^{p}] = 1. 
\end{align*}
Thus, by choosing $\delta = 1- \left(1-\left(\frac{\epsilon}{2}\right)^p\right)^{1/p}>0$, we deduce from the preceding estimate that $\left\| \frac{u+v}{2} \right\|_{X} \leq 1-\delta$.
Hence, $X$ is uniformly convex for $p \geq 2$.\\
\noi We now consider the case $1<p<2$. Set $p'=\frac{p}{p-1}$. \tk{Using \cite[Lemma 5.1, Formula (5.2)]{Chen2022} with $u,v \in \D^{s,p}(\R^N)$, we have
\begin{align*}
    \left\| \frac{u+v}{2} \right\|_{s,p}^{p'} + \left\| \frac{u-v}{2} \right\|_{s,p}^{p'} \leq \left[ \frac{1}{2} \|u\|_{s,p}^p + \frac{1}{2} \|v\|_{s,p}^p \right]^{\frac{1}{p-1}}.
\end{align*}}
Now fix $0<\epsilon_1 \leq 2$, and let $u,v \in \D^{s,p}(\R^N)$ satisfy $\|u\|_{s,p} = 1 = \|v\|_{s,p}$ and $\|u-v\|_{s,p} \geq \epsilon_1$. Define $\delta_1 = 1- \big( 1- (\epsilon_1 /2)^{p'} \big)^{\frac{1}{p'}}>0$. Then $ \left\| \frac{u+v}{2} \right\|_{s,p} \leq 1 - \delta_1$, which shows that $\|\cdot\|_{s,p}$ is a uniformly convex norm. By analogy, the norm $\|u\|_{w_2,p} := \left( \int_{\R^N} w_2 |u|^p\, \dx \right)^{1/p}$ is also uniformly convex. Returning to \eqref{UC}, we observe that $\|u\|_{s,p} \leq 1$ and $\|v\|_{s,p} \leq 1$ and we may assume $\|u-v\|_{s,p} \geq \frac{\epsilon}{2^{1/p}}$. We claim that there exists $\delta_2 >0$ such that 
\begin{align} \label{E1}
    \left\|\frac{u+v}{2} \right\|_{s,p}^p \leq \frac{1- \delta_2}{2} \left(\|u\|_{s,p}^p + \|v\|_{s,p}^p \right). 
\end{align}
We prove the above claim by contradiction.
The proof is divided into two parts.\\
\noi \textbf{Case 1}. Let $\|u\|_{s,p} = 1$ and $\|v\|_{s,p} \leq 1$. Suppose, for contradiction, that claim \eqref{E1} is false. Then there exist $\epsilon_0 >0$ and sequences $\{u_n\},\{v_n\} \subset X$ such that $\|u_n\|_{s,p} = 1,\, \|v_n\|_{s,p} \leq 1$ and $\|u_n-v_n\|_{s,p} \geq \frac{\epsilon_0}{2^{1/p}}$, and satisfying 
\begin{align} \label{E2}
    \left\|\frac{u_n+v_n}{2} \right\|_{s,p}^p \geq \frac{1}{2} \left(1- \frac{1}{n} \right) \left(\|u_n\|_{s,p}^p + \|v_n\|_{s,p}^p \right).
\end{align}
We first prove that $\lim_{n \rightarrow \infty} \|v_n\|_{s,p} =1$. Suppose, toward a contradiction, that $\lim_{n \rightarrow \infty} \|v_n\|_{s,p} <1$. By definition, there exist a subsequence $\{v_{n_l}\}\subset\{v_{n}\}$ and a constant $B<1$ such that $\|v_{n_l}\|_{s,p} \leq B$, for all $l$. Applying the triangle inequality, we obtain
\begin{align} \label{E3}
    \left\|\frac{u_{n_l}+v_{n_l}}{2} \right\|_{s,p}^p \leq \left(\frac{\|u_{n_l}\|_{s,p} + \|v_{n_l}\|_{s,p}}{2} \right)^p \leq \frac{\|u_{n_l}\|_{s,p}^p + \|v_{n_l}\|_{s,p}^p}{2}  \left[\left(\tfrac{1+B}{2} \right)^p / \left(\tfrac{1+B^p}{2} \right) \right],
\end{align}
where the final inequality follows from the monotonicity (increase) of the function $f(x) = \tfrac{(1+x)^p}{1+x^p},\ 1<p<2, \ x \in (0,1)$. Observe that $\left(\tfrac{1+B}{2} \right)^p / \left(\tfrac{1+B^p}{2} \right)<1$ for all $1<p<2$. Consequently, \eqref{E2} contradicts \eqref{E3}. Hence, $\lim_{n \rightarrow \infty} \|v_n\|_{s,p} =1$. 

Define $w_n = \frac{v_n}{\|v_n\|_{s,p}}$. It is straightforward to verify that $\lim\limits_{n \rightarrow \infty}\|v_n - w_n\|_{s,p}=0$. Passing to the limit as $n \rightarrow \infty$ in \eqref{E2} and using $\lim\limits_{n \rightarrow \infty} \|v_n\|_{s,p} =1$, we have
\begin{align*} 
   1 \leq \lim\limits_{n \rightarrow \infty} \left\|\frac{u_{n}+v_{n}}{2} \right\|_{s,p} \leq \lim\limits_{n \rightarrow \infty}\left\|\frac{u_{n}+w_{n}}{2} \right\|_{s,p} \leq 1,
\end{align*}
which implies $\lim_{n \rightarrow \infty}\left\|\frac{u_{n}+w_{n}}{2} \right\|_{s,p}=1$. Since $\|u_n-v_n\|_{s,p} \geq \frac{\epsilon_0}{2^{1/p}}$ for all $n\geq 1$, there exists a positive integer $N_1$ such that $\|u_n-w_n\|_{s,p} \geq \frac{\epsilon_0}{2^{1+1/p}}$ for all $n \geq N_1$. The uniform convexity of the $\|\cdot\|_{s,p}$ norm then guarantees the existence of $\delta_3 >0$ depending on $\epsilon_0$ such that $\left\|\frac{u_{n}+w_{n}}{2} \right\|_{s,p}\leq 1 - \delta_3$ for all $n \geq N_0$. This contradicts $\lim_{n \rightarrow \infty}\left\|\frac{u_{n}+w_{n}}{2} \right\|_{s,p}=1$. Hence, claim \eqref{E1} follows.

\noi \textbf{Case 2}.  Let $\|u\|_{s,p} \leq 1$ and $\|v\|_{s,p} \leq 1$. Without loss of generality, assume that $\|u\|_{s,p} \geq \|v\|_{s,p}>0$, the other case being analogous. Define $ u_1 = \frac{u}{\|u\|_{s,p}}, \ v_1 = \frac{v}{\|u\|_{s,p}} $. Then $\|u_1\|_{s,p}=1, \ \|v_1\|_{s,p}\leq 1 \ \text{and} \ \|u_1 - v_1\|_{s,p} \geq \frac{\epsilon}{2^{1/p}}$. By Case 1, inequality \eqref{E1} holds for $u_1$ and $v_1$, and hence it also holds for $u$ and $v$. Thus, using 
\begin{align*}
    \frac{\|u\|_{s,p}^p +\|v\|_{s,p}^p}{2} \geq \left\|\frac{u-v}{2} \right\|_{s,p}^p \geq \frac{\epsilon^p}{2^{p+1}},
\end{align*}
we get
\begin{align*}
    \left\|\frac{u+v}{2} \right\|_{X} &= \left(\left\|\frac{u+v}{2} \right\|_{s,p}^{p} + \int_{\R^N} w_2 \left|\frac{u+v}{2}\right|^p \dx \right)^{1/p} \\
    & \leq \left( (1- \delta_2)\frac{\|u\|_{s,p}^p +\|v\|_{s,p}^p}{2} + \int_{\R^N} w_2 \left( \frac{|u|^p+|v|^p}{2} \right) \dx \right)^{1/p} \\
    & = \left( \frac{1}{2} \|u\|_{X}^p + \frac{1}{2} \|v\|_{X}^p -\delta_2\frac{\|u\|_{s,p}^p +\|v\|_{s,p}^p}{2} \right)^{1/p} \\
    & \leq \left( 1 -\delta_2 \frac{\epsilon^p}{2^{p+1}} \right)^{1/p} := 1 - \delta,
\end{align*} 
where $\delta = 1- \left( 1 -\delta_2 \frac{\epsilon^p}{2^{p+1}} \right)^{1/p}>0$. Hence, we conclude that the space $(X, \|\cdot\|_X)$ is uniformly convex for $1<p<2$.
\end{proof}
To fix notation, let $X'$ denote the dual space of $X$, and let $\langle \cdot , \cdot \rangle$ denote the corresponding duality pairing.
By the definition of the norm $\|\cdot\|_{X}$, it is straightforward to verify that the functional $W_{2}$ given by
\[W_{2}(\varphi) = \int_{\mathbb{R}^{N}} w_{2}|\varphi|^{p} \dx ,\]
is continuous on $X$. Moreover, $W_{2}$ is continuously Fréchet differentiable on $X$, and its Fréchet derivative is given by
\[\left\langle W'_{2}(\varphi), v \right\rangle =  p \int_{\mathbb{R}^{N}} w_{2}|\varphi|^{p-2} \varphi v\, \dx. \]
Similarly, by means of the fractional Hardy type inequality, one can verify that the functional $W_1$ is of class $C^1$ on $X$, with Fréchet derivative given by
\[\left\langle W'_{1}( \varphi), v \right\rangle =  p \int_{\mathbb{R}^{N}} w_{1}| \varphi|^{p-2} \varphi v\, \dx. \] 
Thus, for $w_1 \in \mathcal{H}_{s,p,0}(\R^N) \text{ and } w_2 \in L^{1}_{\mathrm{loc}}(\R^N)$, the functional $W$ belongs to $C^{1}(X;\R)$, with Fréchet derivative given by
\[\left\langle W'_{}(\varphi), v \right\rangle =  p \int_{\mathbb{R}^{N}} w_{}| \varphi|^{p-2}\varphi v\, \dx. \] 
It is immediate that for $u \in \mathbb{S},\  \langle W'(u), u \rangle = p$, and consequently $W'(u) \neq 0$. Hence, $1$ is a regular value of $W$, where we recall that $\alpha \in \R$ is called a regular value of $W$, if $W'(\varphi) \neq 0$ for all $\varphi$ such that $W(\varphi)=\alpha$. Moreover, by \cite[Example~27.2]{Deimling1985}, $\mathbb{S}$ admits a $C^1$ Banach sub-manifold structure on $X$.

Next, we verify that the functional $J$ satisfies all the hypotheses of Theorem~\ref{Ljusternik}.
\begin{lemma}
The functional J is $C^1$ on $\mathbb{S}$ and the Fr\"{e}chet derivative of J is given by 
\begin{align*}
    \left\langle J'(u),v \right\rangle = p \int_{\mathbb{R}^{N} \times \mathbb{R}^{N}} \frac{|u(x)-u(y)|^{p-2} (u(x)-u(y)) (v(x) -v(y))}{ |x-y|^{N+sp}} \dxy .
\end{align*}
\end{lemma}
We omit the proof as it is straightforward.
\begin{remark}\label{rem2}
We can deduce from \cite[Proposition 6.4.35]{Drabek} that
\[ \|J'(u)\| = \min_{\lambda \in \mathbb{R}} \| J'(u) - \lambda W'(u)\| . \]
Thus $ J'(u_n) \rightarrow 0$  if and only if there exists a sequence $\{\lambda_n \}$ of real numbers such that $J'(u_n) - \lambda_{n} W'(u_n) \rightarrow 0.$
\end{remark}
\begin{definition}
For $\lambda \in \mathbb{R}^{+}$, we define $A_\lambda : X \rightarrow X'$ as
\[A_\lambda = J' + \lambda W'_{2}.\]
\end{definition}
\noi The following lemma is motivated by Szulkin and Willem \cite[Lemma 4.3]{Szulkin1998}.
\begin{lemma} \label{convergence lemma}
If $u_{n} \rightharpoonup u$ in $X$ and $\langle A_{\lambda}(u_{n}), u_{n}-u \rangle \rightarrow 0$, then $ u_{n} \rightarrow u$ in X.
\end{lemma}
\begin{proof}
Clearly, $\langle A_{\lambda}(u_{n})- A_{\lambda}(u), u_{n}-u \rangle \rightarrow 0$. 
We rewrite
\begin{align*}
    \left\langle A_{\lambda}(u_{n})- A_{\lambda}(u), u_{n}-u \right\rangle = B_n + \lambda C_n,
\end{align*}
where $B_n = \langle J'(u_n) - J'(u), u_n -u \rangle $ and $C_n = \langle W'_{2}(u_n) - W'_{2}(u), u_n -u \rangle$.
By Hölder's inequality, we have {\scriptsize
\begin{align*}
 \frac{C_n}{p} &=  \int_{\mathbb{R}^{N}} w_{2} \left(|u_{n}|^{p-2}u_{n} - |u|^{p-2}u\right)(u_{n}-u)\, \dx \\
 &=\int_{\mathbb{R}^{N}} w_{2} (|u_{n}|^{p}+ |u|^{p})\,\dx - \int_{\mathbb{R}^{N}} w_{2}|u_{n}|^{p-2}u_{n}u \,\dx - \int_{\mathbb{R}^{N}} w_{2}|u|^{p-2}u u_{n}\, \dx\\
 &\geq \int_{\mathbb{R}^{N}} w_{2} \left(|u_{n}|^{p}+ |u|^{p}\right)\,\dx - \left( \int_{\mathbb{R}^{N}} w_{2}|u_{n}|^{p} \,\dx \right)^{\frac{p-1}{p}} \left( \int_{\mathbb{R}^{N}} w_{2}|u|^{p}\,\dx \right)^{\frac{1}{p}} -\left( \int_{\mathbb{R}^{N}} w_{2}|u|^{p}\,\dx \right)^{\frac{p-1}{p}} \left( \int_{\mathbb{R}^{N}} w_{2}|u_{n}|^{p}\,\dx \right)^{\frac{1}{p}}\\
 &= \left[ \left( \int_{\mathbb{R}^{N}} w_{2}|u_{n}|^{p} ~\dx \right)^{\frac{p-1}{p}} - \left( \int_{\mathbb{R}^{N}} w_{2}|u|^{p}\,\dx \right)^{\frac{p-1}{p}} \right] \left[ \left( \int_{\mathbb{R}^{N}} w_{2}|u_{n}|^{p} \,\dx \right)^{\frac{1}{p}} - \left( \int_{\mathbb{R}^{N}} w_{2}|u|^{p} \,\dx \right)^{\frac{1}{p}} \right] \geq 0.
 \end{align*}}
Applying Hölder's inequality once again, we obtain \tk{
\begin{align*}
     \frac{B_n}{p} =
    & \|u_n\|_{s,p}^p + \|u\|_{s,p}^p - \int_{\mathbb{R}^{N} \times \mathbb{R}^{N}}  \frac{|u_{n}(x)-u_{n}(y)|^{p-2}(u_{n}(x)-u_{n}(y))(u(x)-u(y))}{|x-y|^{N+sp}}\dxy \\
    &\quad- \int_{\mathbb{R}^{N} \times \mathbb{R}^{N}} \frac{|u(x)-u(y)|^{p-2}(u(x)-u(y)) (u_{n}(x)-u_{n}(y))}{|x-y|^{N+sp}} \dxy\\
    &\geq  \|u_n\|_{s,p}^p + \|u\|_{s,p}^p -\|u_n\|_{s,p}^{p-1} \|u\|_{s,p} -  \|u\|_{s,p}^{p-1} \|u_n\|_{s,p}\\
    &= \left( \|u_n\|_{s,p}^{p-1}- \|u\|_{s,p}^{p-1}\right) \left( \|u_n\|_{s,p}- \|u\|_{s,p}\right)\\
    & \geq 0.
\end{align*}}
Since $\langle A_{\lambda}(u_{n})- A_{\lambda}(u), u_{n}-u \rangle \rightarrow 0$ as $n \rightarrow \infty$ and the sequences $B_n$ and $C_n$ are nonnegative, we deduce that
\begin{align*}
    \lim_{n\rightarrow \infty} B_n = 0 =\lim_{n\rightarrow \infty} C_n.
\end{align*}
This implies
\begin{align*}
  \int_{\mathbb{R}^{N}} w_{2}|u_{n}|^{p}\,\dx \rightarrow \int_{\mathbb{R}^{N}} w_{2}|u|^{p} \,\dx, \quad \text{ and }\quad    \|u_n\|_{s,p} \rightarrow \|u\|_{s,p} \quad \text{ as } n \rightarrow \infty.
\end{align*}
Consequently, $\|u_{n}\|_{X} \rightarrow \|u\|_{X}$ and therefore $u_{n} \rightarrow u $ in X.
\end{proof}
\begin{lemma}
For $w_{1} \in \mathcal{H}_{s,p,0}(\mathbb{R}^{N})$, the map $W'_{1} : X \rightarrow X' $ is compact.
\end{lemma}

\begin{proof}
Let $u_{n} \rightharpoonup u$ in $X$ and $v \in X$. For $w_{1} \in \mathcal{H}_{s,p,0}(\mathbb{R}^{N})$, by Theorem \ref{H1} we have 
\begin{equation} \label{Sobolev type inequality}
    \| w_{1}^{\frac{1}{p}} u \|_{p}  \leq C \|w_{1}\|^{\frac{1}{p}}_{\mathcal{H}_{s,p}(\R^N)} \|u\|_{s,p},
\end{equation}  
where the constant $C>0$ depends on $N,s $ and $p$ only. Applying Hölder's inequality yields
\begin{align*}
    \left|\left\langle W_{1}'(u_n) - W'_{1}(u), v \right\rangle\right|
    &\leq ~\int_{\mathbb{R}^{N}} w_{1} \left| |u_{n}|^{p-2}u_{n} - |u|^{p-2}u \right| |v| \,\dx  \\
    &\leq  \left(\int_{\mathbb{R}^{N}} w_{1} \left||u_{n}|^{p-2}u_{n} -|u|^{p-2}u \right|^{\frac{p}{p-1}} \dx \right)^{\frac{p-1}{p}}  \left(\int_{\mathbb{R}^{N}} w_{1} |v|^{p}\, \dx \right)^{\frac{1}{p}} \\
    &\leq C~ \|w_{1}\|^{\frac{1}{p}}_{\mathcal{H}_{s,p}} \|v\|_{s,p} \left(\int_{\mathbb{R}^{N}} w_{1} \left||u_{n}|^{p-2}u_{n} -|u|^{p-2}u\right|^{\frac{p}{p-1}}\, \dx \right)^{\frac{p-1}{p}}.
\end{align*}
Thus
\begin{align}\label{4.20}
    \left\|W'_{1}(u_{n}) - W'_{1}(u)\right\| \leq C ~\|w_{1}\|^{\frac{1}{p}}_{\mathcal{H}_{s,p}} \left(\int_{\mathbb{R}^{N}} w_{1} \left||u_{n}|^{p-2}u_{n} -|u|^{p-2}u\right|^{\frac{p}{p-1}} \,\dx \right)^{\frac{p-1}{p}}.
\end{align}
It suffices to prove that the right hand side integral in the above estimate tends to zero as $n$ goes to infinity.
Let $\epsilon>0$ and $w_{\epsilon} \in C_{c}^{\infty}(\mathbb{R}^{N})$ which will be specified later. Applying the triangle inequality, we obtain
\begin{align}\label{Integral with w}
&\int_{\mathbb{R}^{N}} w_{1} \left||u_{n}|^{p-2}u_{n} -|u|^{p-2}u\right|^{\frac{p}{p-1}}\, \dx \notag\\
&\leq \int_{\mathbb{R}^{N}} |w_{\epsilon}| \left||u_{n}|^{p-2}u_{n} -|u|^{p-2}u\right|^{\frac{p}{p-1}}\, \dx + \int_{\mathbb{R}^{N}} |w_{1} -w_{\epsilon}| \left||u_{n}|^{p-2}u_{n} -|u|^{p-2}u\right|^{\frac{p}{p-1}}\, \dx.
\end{align}
Define $K := \sup \left\{\|u_{n}\|_{s,p}^{p} + \|u\|_{s,p}^{p}: n \in \mathbb{N}\right\}$. Then $K$ is finite since $\{u_{n}\}$ is bounded in $X$. For the second integral,
\begin{align}\label{4.21}
\int_{\mathbb{R}^{N}} |w_{1} -w_{\epsilon}| \left| |u_{n}|^{p-2}u_{n} -|u|^{p-2}u\right|^{\frac{p}{p-1}} \,\dx
&\leq \int_{\mathbb{R}^{N}} |w_{1} -w_{\epsilon}| \left(|u_{n}|^{p-1} +|u|^{p-1}\right)^{\frac{p}{p-1}}\, \dx \notag \\
&\leq 2^{\frac{1}{p-1}} \left( \int_{\mathbb{R}^{N}} |w_{1} -w_{\epsilon}| |u_{n}|^{p} \,\dx + \int_{\mathbb{R}^{N}} |w_{1} -w_{\epsilon}||u|^{p} \,\dx \right) \notag \\
    &\leq 2^{\frac{1}{p-1}} C  \|w_{1} - w_{\epsilon}\|_{\mathcal{H}_{s,p}} \left(\|u_{n}\|_{s,p}^{p}+ \|u\|_{s,p}^{p}\right) \notag\\
    &\leq 2^{\frac{1}{p-1}} C K \|w_{1} - w_{\epsilon}\|_{\mathcal{H}_{s,p}}.
\end{align}
As $w_{1} \in \mathcal{H}_{s,p,0}(\mathbb{R}^{N})$, we can choose $w_{\epsilon} \in C_{c}^{\infty}(\mathbb{R}^{N})$ satisfying
\begin{align}\label{4.22}
    2^{\frac{1}{p-1}} K \|w_{1} - w_{\epsilon}\|_{\mathcal{H}_{s,p}} < \frac{\epsilon}{2C}.
\end{align}
From \eqref{4.21} and \eqref{4.22}, we deduce that the second integral in \eqref{Integral with w} is less than $\frac{\epsilon}{2}$. 
Since $X$ is compactly embedded in $L^{p}_{\mathrm{loc}}(\mathbb{R}^N)$, the first integral converges to zero up to a subsequence $\{u_{n_k}\} \subset \{u_n\}$. 
Hence, there exists $k_0 \in \mathbb{N}$ such that
\begin{align*}
\int_{\mathbb{R}^{N}} w_{1} \left||u_{n_{k}}|^{p-2}u_{n_{k}} -|u|^{p-2}u\right|^{\frac{p}{p-1}}\, \dx < \epsilon, \quad\forall ~k>k_{0}.
\end{align*}
We conclude from the uniqueness of sequential limits that
\begin{align}\label{4.24}
\int_{\mathbb{R}^{N}} w_{1} \left||u_{n}|^{p-2}u_{n} -|u|^{p-2}u\right|^{\frac{p}{p-1}}\, \dx ~\rightarrow ~0\quad \text{ as }n \rightarrow \infty.
\end{align}
From \eqref{4.20} and \eqref{4.24}, we get the desired result.
\end{proof}

Next, we prove that the functional $J$ satisfies the Palais–Smale condition on $\mathbb{S}$.
\begin{proposition}
The functional $J$ satisfies the Palais-Smale (PS) condition on $\mathbb{S}$.
\end{proposition}
 \begin{proof}
Let $\{u_n\} \subset \mathbb{S}$ be a sequence such that $J(u_n) \to \lambda$ and
$J'(u_n) \to 0
\text{ as } n \to \infty$.
Then, by Remark~\ref{rem2}, there exists a sequence $\{\lambda_n\} \subset \mathbb{R}$ such that
\begin{equation} \label{Palais smale condition}
   J'(u_n) - \lambda_{n} W'(u_n) \rightarrow 0 \quad \text{ in }X' \text{ as } n \rightarrow \infty. 
\end{equation}
Since $\{J(u_n)\}$ is bounded in $\R$ and $\int_{\mathbb{R}^N} w\,|u_n|^p \,\dx > 0,$
together with the inequality
\begin{align}\label{L1}
\int_{\mathbb{R}^N} w_2\,|u_n|^p \,\dx
&< \int_{\mathbb{R}^N} w_1\,|u_n|^p \,\dx
\le C\,\|w_1\|_{\mathcal{H}_{s,p}}\,\|u_n\|_{s,p}^p,
\end{align}
we deduce that $\{W_2(u_n)\}$ is bounded in $\mathbb{R}$. Consequently, $\{u_n\}$ is bounded in $X$. As $X$ is reflexive, there exists a subsequence (still denoted by $\{u_n\}$) and some $u \in X$ such that $u_n \rightharpoonup u $ in $X$. Since $X$ is continuously embedded in $\D^{s,p}(\R^N)$, $W_1$ is compact on $X$. Thus, $W_1(u_n) \rightarrow W_1(u)$ in $\R$. 

Now Fatou's Lemma yields
\begin{align} \label{L2}
    \int_{\mathbb{R}^{N}} w_{2} |u|^{p} \dx \leq  \liminf_{n} \int_{\mathbb{R}^{N}} w_{1} |u_n|^{p} \dx - 1 = \int_{\mathbb{R}^{N}} w_{1} |u|^{p} \dx - 1.
\end{align}
Thus $\int_{\mathbb{R}^{N}} w |u|^{p}\, \dx \geq 1 $, and hence $ u \neq 0.$ Moreover, $\lambda_n \to \lambda$ as $n \to \infty$, since
\begin{align*}
    p \left(J(u_n) - \lambda_n \right)  = \left\langle J'(u_n) - \lambda_n W'(u_n), u_n \right\rangle ~ \ra ~0, \text{ as } n \ra \infty.
\end{align*}
We rewrite \eqref{Palais smale condition} as
\begin{align*}
    A_{\lambda_n}(u_n) - \lambda_n W'_{1}(u_n) \rightarrow 0 ~ \text{as}~n \rightarrow \infty.
\end{align*}
Since $\lambda_n \to \lambda$, it follows that $A_{\lambda_n}(u_n) - A_{\lambda}(u_n) \rightarrow 0$ in $X'$. The compactness of $W'_1$ implies that $A_{\lambda}(u_n)$ converges strongly in $X'$, and therefore $\langle A_{\lambda}(u_n), u_{n} - u \rangle \rightarrow 0$. As $u_n \rightharpoonup u$ in $X$, Lemma~\ref{convergence lemma} yields $u_{n} \rightarrow u$ in $X$.
 \end{proof} 
Next, we state the following Lemma:
\begin{lem}[{{\cite[Lemma 5.9]{{anoop-p}}}}]
The set $\Gamma_{n}$ is non-empty for each $n \in \mathbb{N}$.
\end{lem}

Finally, we prove the existence of an infinite set of eigenvalues for (\ref{Weighetd eigenvalue problem}) by employing the Ljusternik-Schnirelmann theorem on $C^1$-manifold.
\begin{proof}[Proof of theorem \ref{Infinite eigenvalue}]
Recall that the functional $J$ and the set $\mathbb{S}$ satisfy all the conditions of Theorem \ref{Ljusternik}. Therefore, we get $\gamma(K_{\lambda_{j}}) \geq 1$ for each $j \in \mathbb{N}$. Thus $K_{\lambda_{j}} \neq \emptyset$ and hence there exists $u_{j} \in \mathbb{S}$ such that $J'(u_{j}) = 0$ and $J(u_{j}) = \lambda_{j}$. Therefore, $\lambda_{j}$ is an eigenvalue and $u_{j}$ is the corresponding eigenfunction for (\ref{Weighetd eigenvalue problem}).
Recall that X is separable \cite[Lemma 2.1]{Cui} and hence X admits a bi-orthogonal system $\{ e_{m}, e_{m}^{*}\}$ such that 
\begin{align*}
    \overline{ \{ e_{m} : m \in \mathbb{N}  \}} = X,\quad e_{m}^{*} \in X',\quad \langle {e_{n}},e_{m}^{*}\rangle = \delta_{n,m} \\
    \langle e_{m}^{*}, x \rangle = 0,\quad \forall m \implies x=0.
\end{align*}
Let $E_{n} = span \{ e_{1}, e_{2}, ... ,e_{n} \}$ and let $E_{n}^{\perp} = \overline{span \{ e_{n+1}, e_{n+2}, ...  \}}$. Since  $E_{n-1}^{\perp}$ is of co-dimension $(n-1)$, for any $A \in \Gamma_{n}$ we have $A \cap E_{n-1}^{\perp} \neq \emptyset.$ Let 
\begin{align*}
   \mu_{n} = \inf_{A \in \Gamma_{n}} \sup_{A \cap E_{n-1}^{\perp}} J(u),\quad n=1,2,... 
\end{align*}
We now prove that $\mu_n \to \infty$ as $n \to \infty$. On contrary suppose that $\{ \mu_{n} \}$ is bounded, then there exists $u_{n} \in E_{n-1}^{\perp} \cap \mathbb{S}$ such that $\mu_{n} \leq J(u_{n}) <c$, for some constant $c >0$. Since $u_{n} \in \mathbb{S}$, the sequence $\{u_{n}\}$ is bounded in $X$ by estimate \eqref{L1}. Hence, there exists $u \in X$ such that $u_{n} \rightharpoonup u$ in $X$. By the choice of the biorthogonal system $\{e_m, e_m^*\}$, we have for each $m$, $\langle e_{m}^{*}, u_{n} \rangle \rightarrow 0$ as $n \rightarrow \infty$. Thus $u_{n} \rightharpoonup 0$ in X and hence $u = 0$, a contradiction to $ \int_{\mathbb{R}^{N}} w|u|^{p} \dx \geq 1$ (See the conclusion of \eqref{L2}). Therefore, $\mu_{n} \rightarrow \infty$ as $n \ra \infty$ and $\lambda_{n} \rightarrow \infty \ \text{as } n \ra \infty, \ \text{since }\mu_{n} \leq \lambda_{n}$. Finally, the first eigenvalue is simple by Lemma~\ref{Simple} and principal by Lemma~\ref{Positive}.  
This completes the proof.
\end{proof} 